\documentclass[12pt,a4paper]{article}
\usepackage{latexsym,amssymb,amsfonts,amsmath,amsthm,nccmath,float,enumitem,setspace,bm,authblk,appendix}
\usepackage[usenames,dvipsnames]{xcolor}
\usepackage[hidelinks]{hyperref}
\usepackage[margin=2cm]{geometry}

\setlist{topsep=3pt,partopsep=0pt,itemsep=1pt,parsep=0pt}

\hypersetup{
    colorlinks,
    linkcolor={red!80!black},
    citecolor={blue!80!black},
    urlcolor={blue!80!black}
}

\newtheorem{Theorem}{Theorem}[section]

\newtheorem{Corollary}[Theorem]{Corollary}

\newtheorem{Remark}[Theorem]{Remark}
\newtheorem{Lemma}[Theorem]{Lemma}

\newtheorem{Claim}{Claim}

\def \leq {\leqslant}
\def \geq {\geqslant}

\let\oldproofname=\proofname
\renewcommand{\proofname}{\rm\bf{\oldproofname}}

\numberwithin{equation}{section}

\allowdisplaybreaks[4]

\begin{document}

\title{The structure of maximal non-trivial $d$-wise intersecting uniform families with large sizes}

\author[a]{Menglong Zhang}
\author[a]{Tao Feng}
\affil[a]{School of Mathematics and Statistics, Beijing Jiaotong University, Beijing, 100044, P.R. China}
\affil[ ]{mlzhang@bjtu.edu.cn; tfeng@bjtu.edu.cn}

\date{}
\maketitle

\footnotetext{Supported by NSFC under Grants 12271023 and 11871095}


\begin{abstract}
For a positive integer $d\geq 2$, a family $\mathcal F\subseteq \binom{[n]}{k}$ is said to be $d$-wise intersecting if $|F_1\cap F_2\cap \dots\cap F_d|\geq 1$ for all $F_1, F_2, \dots ,F_d\in \mathcal F$.
A $d$-wise intersecting family $\mathcal F\subseteq \binom{[n]}{k}$ is called maximal if $\mathcal F\cup\{A\}$ is not $d$-wise intersecting for any $A\in\binom{[n]}{k}\setminus\mathcal F$. We provide a refinement of O'Neill and Verstra\"{e}te's Theorem about the structure of the largest and the second largest maximal non-trivial $d$-wise intersecting $k$-uniform families. We also determine the structure of the third largest and the fourth largest maximal non-trivial $d$-wise intersecting $k$-uniform families for any $k>d+1\geq 4$, and the fifth largest and the sixth largest maximal non-trivial $3$-wise intersecting  $k$-uniform families for any $k\geq 5$, in the asymptotic sense. Our proofs are applications of the $\Delta$-system method.
\end{abstract}



\section{Introduction}

For $0<a<b$, let $[a,b]$ denote the set of integers from $a$ to $b$. Let $n$ and $k$ be integers with $1\leq k\leq n$. Write $[n]=\{1,2,\ldots,n\}$.
Denote by $2^{[n]}$ and $\binom{[n]}{k}$ the power set and the family of all $k$-subsets of $[n]$, respectively.

For an integer $d\geq 2$, a family $\mathcal F\subseteq 2^{[n]}$ is said to be {\em $d$-wise intersecting} if $|F_1\cap F_2\cap \dots\cap F_d|\geq 1$ for all $F_1, F_2, \dots ,F_d\in \mathcal F$.
Every $(d+1)$-wise intersecting family is also a $d$-wise intersecting family. A $2$-wise intersecting family is often simply called an {\em intersecting family}. A $d$-wise intersecting family is called {\em trivial} if all of its members contain a common element, and {\em non-trivial} otherwise. For $1\leq k\leq n$, a family ${\cal F}\subseteq 2^{[n]}$ is said to be {\em $k$-uniform} if every member of $\cal F$ contains exactly $k$ elements, i.e., $\mathcal F\subseteq \binom{[n]}{k}$. A $d$-wise intersecting $k$-uniform family $\mathcal F$ is called {\em maximal} if $\mathcal F\cup\{A\}$ is not $d$-wise intersecting for every $A\in \binom{[n]}{k}\setminus\mathcal F$.

Two families $\mathcal F,\mathcal F'\subseteq 2^{[n]}$ are said to be {\em isomorphic}, denoted by ${\cal F}\cong{\cal F}'$, if there exists a permutation $\pi$ on $[n]$ such that $\{\{\pi(x):x\in F\}:F\in{\cal F}\}={\cal F}'$.

\subsection{Brief historical review}

The celebrated Erd\H{o}s-Ko-Rado theorem \cite{EKR} determines the size and the structure of the largest intersecting uniform families. It states that if $n\geq2k$ and $\mathcal F\subseteq \binom{[n]}{k}$ is an intersecting family, then $|\mathcal F|\leq \binom{n-1}{k-1}$; moreover, for $n>2k$, the equality holds if and only if $\mathcal F\cong\{F\in\binom{[n]}{k}:1\in F\}$ is a trivial intersecting family.


The structure of maximal non-trivial intersecting uniform families with large sizes has been investigated by many authors. Hilton and Milner \cite{HM} examined the structure of the largest non-trivial intersecting uniform families. They showed that if $n>2k\geq4$ and $\mathcal F\subseteq \binom{[n]}{k}$ is a non-trivial intersecting family with the largest size, then $|\mathcal F|=\binom{n-1}{k-1}-\binom{n-k-1}{k-1}+1$ and $\mathcal F\cong\{[2,k+1]\}\cup\{F\in\binom{[n]}{k}:1\in F,|F\cap[2,k+1]|\geq 1\}$ or $\mathcal F\cong\{F\in\binom{[n]}{3}:|F\cap[3]|\geq 2\}$. Especially, if $\mathcal F\subseteq\binom{[n]}{2}$ is a maximal intersecting family, then $\mathcal F\cong \{\{1,i\}:i\in[2,n]\}$ or $\binom{[3]}{2}$, and hence $\binom{[3]}{2}$ is the only non-trivial intersecting $2$-uniform family. Han and Kohayakawa \cite{HK} determined the structure of the second largest maximal non-trivial intersecting $k$-uniform families for any $n>2k\geq 6$, and as an open problem they asked what is the size of the third largest maximal non-trivial intersecting $k$-uniform families. Kostochka and Mubayi \cite{KM} gave an asymptotic solution to this problem as a corollary of a more general result. They characterized the structure of maximal intersecting families $\mathcal F\subseteq\binom{[n]}{k}$ with $|\mathcal F|\geq \binom{n-1}{k-1}-\binom{n-k+1}{k-1}+4\binom{n-k-2}{k-3}+\binom{n-k-3}{k-5}+2$ for any $k\geq 4$ and all large enough $n$. Kupavskii \cite{Ku} removed the sufficiently large requirement for $n$ and extended Kostochka and Mubayi's result to any $n>2k\geq 10$. Very recently, the question of Han and Kohayakawa \cite{HK} was completely settled by Huang and Peng \cite{HP} for any $n>2k\geq 8$.

Frankl \cite{F76a,F87} generalized the Erd\H{o}s-Ko-Rado theorem to $d$-wise intersecting uniform families.
He showed that if $\mathcal F\subseteq \binom{[n]}{k}$ is a $d$-wise intersecting family with $d\geq 2$ and $n\geq\frac{dk}{d-1}$, then $|\mathcal F|\leq\binom{n-1}{k-1}$; furthermore, except for the case of $d=2$ and $n=2k$, the equality holds if and only if $\mathcal F\cong\{F\in\binom{[n]}{k}:1\in F\}$ is a trivial intersecting family (cf. \cite{Kamat}).
For the case of non-trivial $d$-wise intersecting families, Hilton and Milner \cite{HM} showed that if a non-trivial $d$-wise intersecting $k$-uniform family exists, then $d\leq k$. For $d=k$, they proved that if $n\geq k+1$ and $\mathcal F\subseteq\binom{[n]}{k}$ is a maximal non-trivial $k$-wise intersecting family, then $\mathcal F\cong\binom{[k+1]}{k}$.
They conjectured that if $n$ is sufficiently large, then a non-trivial $d$-wise intersecting family $\mathcal F\subseteq\binom{[n]}{k}$ with the largest size must be one of the following two families up to isomorphism:
\begin{align*}
\mathcal A(k,d)=\left\{A\in\binom{[n]}{k}:|A\cap[d+1]|\geq d\right\}
\end{align*}
and
\begin{align*}
\mathcal B(k,d)=\left\{A\in\binom{[n]}{k}:[d-1]\subseteq A, A\cap[d,k+1]\neq\emptyset\right\}\cup\binom{[k+1]}{k}.
\end{align*}
O'Neill and Verstra\"{e}te \cite{OV} provided an asymptotic solution to this conjecture and established a stability theorem. Let
\begin{align*}
\mathcal C(k,d)= & \left\{A\in\binom{[n]}{k}:[d-1]\subseteq A,A\cap[d,k]\neq\emptyset\right\}\\
& \cup \left\{A\in\binom{[n]}{k}:|A\cap[d-1]|=d-2,[d,k]\subseteq A \right\},
\end{align*}
which is a non-trivial $d$-wise intersecting family.

\begin{Theorem}\label{thm:non_d_wise_intersecting}{\rm \cite{OV}}
Let $k>d\geq 3$ and $n>n_0(k,d)=d+e(k-d)(k^22^k)^{2^k}$. If $\mathcal F\subseteq \binom{[n]}{k}$ is a maximal non-trivial $d$-wise intersecting family, then
$$|\mathcal F|\leq\max\{|\mathcal A(k,d)|,|\mathcal B(k,d)|\}.$$
Furthermore, if $2d\geq k$ and $|\mathcal F|>\min\{|\mathcal A(k,d)|,|\mathcal B(k,d)|\}$, then ${\cal F}\cong \mathcal A(k,d)$ or $\mathcal B(k,d)$; if $2d<k$ and $|\mathcal F|>|\mathcal C(k,d)|$, then $\mathcal F\cong \mathcal B(k,d)$.
\end{Theorem}



\subsection{Our contribution}\label{sec:contributions}

This paper is devoted to examining the structure of maximal non-trivial $d$-wise intersecting $k$-uniform families with large sizes for any $d\geq 3$. For two functions $f$ and $g$, the notation $f\sim g$ denotes that $f=(1+o(1))g$.

First of all, we describe some maximal non-trivial $d$-wise intersecting families, which will be used as extremal families in our main theorems (see Theorems \ref{thm:maximal_intersecting_d+2} and \ref{thm:maximal_intersecting_d+1}).

For $k\geq d+1\geq4$ and $2\leq l\leq k-d+2$, write
\begin{equation}
\begin{aligned}
\mathcal H(k,d,l)= &\left\{F\in\binom{[n]}{k}:[d-1]\subseteq F,\, F\cap[d,d+l-1]\neq\emptyset\right\}\\
&\cup\left\{F\in\binom{[n]}{k}:|[d-1]\cap F|=d-2,\, [d,d+l-1]\subseteq F\right\}.\notag
\end{aligned}
\end{equation}
It is readily checked that $\mathcal H(k,d,l)$ is non-trivial $d$-wise intersecting and \begin{equation}\label{equ:H(k,d,l)}
\begin{aligned}
|\mathcal H(k,d,l)|=\binom{n-d+1}{k-d+1}-\binom{n-d-l+1}{k-d+1}+(d-1)\binom{n-d-l+1}{k-d-l+2}\sim c\binom{n}{k-d},
\end{aligned}
\end{equation}
where $c=l$ when $l>2$ and $c=d+1$ when $l=2$.

\begin{Remark}\label{rek:H-ABC}
$\mathcal H(k,d,2)=\mathcal A(k,d)$, $\mathcal H(k,d,k-d+2)=\mathcal B(k,d)$ and $\mathcal H(k,d,k-d+1)=\mathcal C(k,d)$.
\end{Remark}

For $k\geq d+1\geq4$, write
\begin{align*}
\mathcal G(k,d)= &\left\{F\in\binom{[n]}{k}:[d-1]\subseteq F,\, F\cap[d,k-1]\neq\emptyset\right\}\cup\{[2,k]\cup\{i\}:i\in[k+2,n]\}\\
&\cup\left\{F\in\binom{[n]}{k}:[d-1]\cup[k,k+1]\subseteq F\right\}\cup\{[2,k-1]\cup\{k+1,i\}:i\in[k+2,n]\}\\
&\cup\left\{F\in\binom{[n]}{k}:|[d-1]\cap F|=d-2,\, [d,k+1]\subseteq F\right\}.
\end{align*}
It is readily checked that $\mathcal G(k,d)$ is non-trivial $d$-wise intersecting and
\begin{align}\label{equ:G(k,d)}
|\mathcal G(k,d)|&=\binom{n-d+1}{k-d+1}-\binom{n-k+1}{k-d+1}+\binom{n-k-1}{k-d-1}+2(n-k)+d-3 \sim(k-d)\binom{n}{k-d}.
\end{align}

\begin{Remark}\label{rek:H-G}
$\mathcal G(d+1,d)$ is isomorphic to $\mathcal H(d+1,d,3)$ under the action of the permutation $\pi=(1\ d)$ on $[n]$ that interchanges the elements $1$ and $d$ and fixes all the other elements of $[n]$.
\end{Remark}

For $k\geq d+2$ and $d=3$, write
\begin{align*}
\mathcal S(k,3)= &\left\{F\in\binom{[n]}{k}:[2]\subseteq F,\, F\cap[3,k-1]\neq\emptyset\right\}
\cup\left\{F\in\binom{[n]}{k}:[2]\cup[k,k+2]\subseteq F\right\}\\
&\cup\left(\bigcup_{i=1}^{2}\left\{F\in\binom{[n]}{k}:\{i\}\cup [3,k-1]\subseteq F,\, |[k,k+2]\cap F|=2\right\}\right),\notag
\end{align*}
and
\begin{align*}
\mathcal S_1(k,3)=&\left\{F\in\binom{[n]}{k}:[2]\subseteq F,\,F\cap[3,k-1]\neq\emptyset\right\} \cup\left\{F\in\binom{[n]}{k}:[2]\cup[k,k+2]\subseteq F\right\}\\
&\cup\left\{[2,k]\cup\{i\}:i\in[k+1,n]\right\}\, \cup\, \left\{[2,k+2]\setminus\{k\},\, [k+1]\setminus\{2\},\, [k+2]\setminus\{2,k+1\}\right\}.\notag
\end{align*}
It is readily checked that $\mathcal S(k,3)$ and $\mathcal S_1(k,3)$ are both non-trivial $3$-wise intersecting,
\begin{equation}\label{equ:S(k,3)}
\begin{aligned}
|\mathcal S(k,3)|=\binom{n-2}{k-2}-\binom{n-k+1}{k-2}+\binom{n-k-2}{k-5}+6\sim(k-3)\binom{n}{k-3} \text{\ \  and }
\end{aligned}
\end{equation}
\begin{equation}\label{equ:S1(k,3)}
\begin{aligned}
|\mathcal S_1(k,3)|=\binom{n-2}{k-2}-\binom{n-k+1}{k-2}+\binom{n-k-2}{k-5}+n-k+3\sim(k-3)\binom{n}{k-3}.
\end{aligned}
\end{equation}
For $k\geq d+3$ and $d=3$, write
\begin{equation}
\begin{aligned}
\mathcal S_2(k,3)=&\left\{F\in\binom{[n]}{k}:[2]\subseteq F,F\cap[3,k-1]\neq\emptyset\right\}\cup\left\{F\in\binom{[n]}{k}:[2]\cup[k,k+3]\subseteq F\right\}\\
&\cup\left\{[2,k+1],\, [2,k-1]\cup\{k+2,k+3\},\, [2,k]\cup\{k+3\},\, [2,k+2]\setminus\{k\}\right\}\\
&\cup\left\{[3,k]\cup\{1,k+2\},\, [3,k-1]\cup\{1,k+1,k+3\}\right\},\notag
\end{aligned}
\end{equation}
and
\begin{equation}
\begin{aligned}
\mathcal S_3(k,3)=&\left\{F\in\binom{[n]}{k}:[2]\subseteq F,F\cap[3,k-1]\neq\emptyset\right\}\cup\left\{F\in\binom{[n]}{k}:[2]\cup[k,k+3]\subseteq F\right\}\\
&\cup\{[2,k]\cup\{k+2\},\, [2,k+2]\setminus\{k\}, \, [2,k-1]\cup\{k+1,k+3\}\}\\
&\cup\{[3,k+1]\cup\{1\},\, [3,k-1]\cup\{1,k+1,k+2\}, \, [3,k-1]\cup\{1,k+2,k+3\}\}.\notag
\end{aligned}
\end{equation}
It is readily checked that $\mathcal S_2(k,3)$ and $\mathcal S_3(k,3)$ are both non-trivial $3$-wise intersecting and
\begin{align}\label{equ:S2(k,3)_S3(k,3)}
|\mathcal S_2(k,3)|=|\mathcal S_3(k,3)|&=\binom{n-2}{k-2}-\binom{n-k+1}{k-2}+\binom{n-k-3}{k-6}+6\sim(k-3)\binom{n}{k-3}.
\end{align}

Our main theorems are as follows:

\begin{Theorem}\label{thm:maximal_intersecting_d+2}
Let $k\geq d+2\geq 5$ and $n>n_1(k,d)=d+2(k-d)^2(k^{k-d}-1)^kk!$.
If $\mathcal F\subseteq \binom{[n]}{k}$ is a maximal non-trivial $d$-wise intersecting family and $|\mathcal F|>(k-d-\frac{1}{2})\binom{n-d}{k-d}$, then the following hold.
\begin{enumerate}
\item[$(1)$] If $k=5$ and $d=3$, then $\mathcal F\cong\mathcal H(k,d,2)$, $\mathcal H(k,d,k-d)$, $\mathcal H(k,d,k-d+1)$, $\mathcal H(k,d,k-d+2)$, $\mathcal G(k,d)$, $\mathcal S(k,3)$ or $\mathcal S_1(k,3)$.

\item[$(2)$]  If $k\in\{6,7\}$ and $d=3$, then $\mathcal F\cong\mathcal H(k,d,2)$, $\mathcal H(k,d,k-d)$, $\mathcal H(k,d,k-d+1)$, $\mathcal H(k,d,k-d+2)$, $\mathcal G(k,d)$, $\mathcal S(k,3)$, $\mathcal S_1(k,3)$, $\mathcal S_2(k,3)$ or $\mathcal S_3(k,3)$.

\item[$(3)$]  If $k>7$ and $d=3$, then $\mathcal F\cong\mathcal H(k,d,k-d)$, $\mathcal H(k,d,k-d+1)$, $\mathcal H(k,d,k-d+2)$, $\mathcal G(k,d)$, $\mathcal S(k,3)$, $\mathcal S_1(k,3)$, $\mathcal S_2(k,3)$ or $\mathcal S_3(k,3)$.

\item[$(4)$]  If $d+2\leq k\leq 2d+1$ and $d>3$, then $\mathcal F\cong\mathcal H(k,d,2)$, $\mathcal H(k,d,k-d)$, $\mathcal H(k,d,k-d+1)$, $\mathcal H(k,d,k-d+2)$ or $\mathcal G(k,d)$.

\item[$(5)$]  If $k>2d+1$ and $d>3$, then $\mathcal F\cong\mathcal H(k,d,k-d)$, $\mathcal H(k,d,k-d+1)$, $\mathcal H(k,d,k-d+2)$ or $\mathcal G(k,d)$.
\end{enumerate}
\end{Theorem}

\begin{Theorem}\label{thm:maximal_intersecting_d+1}
Let $d\geq 3$ and $n>(d+1)^2$. If $\mathcal F\subseteq \binom{[n]}{d+1}$ is a maximal non-trivial $d$-wise intersecting family and $|\mathcal F|>3d(d+1)$, then $\mathcal F$ is isomorphic to one of $\mathcal H(d+1,d,2)$ and $\mathcal H(d+1,d,3)$.
\end{Theorem}

The proof of Theorem \ref{thm:maximal_intersecting_d+2} is presented in Section \ref{sec:2} by using the $\Delta$-system method. The proof of Theorem \ref{thm:maximal_intersecting_d+1} appears in Section \ref{sec:3}. Section \ref{sec:inequalities} compares the sizes of extremal families given in Theorems \ref{thm:maximal_intersecting_d+2} and \ref{thm:maximal_intersecting_d+1}. As a corollary, we provide a refinement of Theorem \ref{thm:non_d_wise_intersecting}.

\begin{Corollary}\label{cor:first_second_max_inter}
Let $k>d\geq 3$. For any $n>n_2(k,d)$, where
\[ n_2(k,d)=
\begin{cases}
(d+1)^2,& \text{if } k=d+1;\\
d+2(k-d)^2(k^{k-d}-1)^kk!,& \text{if } k\geq d+2,
\end{cases} \]
\begin{enumerate}
\item[$(1)$] if $\mathcal F\subseteq \binom{[n]}{k}$ is the largest non-trivial $d$-wise intersecting family, then $\mathcal F\cong \mathcal H(k,d,2)$ when $k\leq 2d-1$ and $\mathcal F\cong \mathcal H(k,d,k-d+2)$ when $k\geq2d$;
\item[$(2)$] if $\mathcal F\subseteq \binom{[n]}{k}$ is the second largest maximal non-trivial $d$-wise intersecting family, then $\mathcal F\cong \mathcal H(k,d,k-d+2)$ when $k\leq2d-1$, $\mathcal F\cong \mathcal H(k,d,2)$ when $k=2d$ and $\mathcal F\cong \mathcal H(k,d,k-d+1)$ when $k\geq2d+1$.
\end{enumerate}
\end{Corollary}

Let $\mathcal F\subseteq\binom{[n]}{k}$ be a maximal non-trivial $d$-wise intersecting family. By Corollary \ref{cor:first_second_max_inter}(1) and Remark \ref{rek:H-ABC}, if $n>n_2(k,d)$, then $|\mathcal F|\leq\max\{|\mathcal H(k,d,2)|,|\mathcal H(k,d,k-d+2)|\}=\max\{|\mathcal A(k,d)|$, $|\mathcal B(k,d)|\}$. Furthermore, Corollary \ref{cor:first_second_max_inter}  and Remark \ref{rek:H-ABC} imply that if $k\leq 2d-1$ and $|\mathcal F|>|\mathcal B(k,d)|$, then $\mathcal F\cong\mathcal A(k,d)$; if $k=2d$ and $|\mathcal F|>|\mathcal A(k,d)|$, then $\mathcal F\cong\mathcal B(k,d)$; if $k\geq 2d+1$ and $|\mathcal F|>|\mathcal C(k,d)|$, then $\mathcal F\cong\mathcal B(k,d)$. Since $n_0(k,d)>n_2(k,d)$, we obtain Theorem \ref{thm:non_d_wise_intersecting} from Corollary \ref{cor:first_second_max_inter}.

As corollaries of Theorems \ref{thm:maximal_intersecting_d+2} and \ref{thm:maximal_intersecting_d+1}, in Section \ref{sec:inequalities}, we determine the structure of the third largest and the fourth largest maximal non-trivial $d$-wise intersecting $k$-uniform families for any $k>d+1\geq 4$, and the fifth largest and the sixth largest maximal non-trivial $3$-wise intersecting  $k$-uniform families for any $k\geq 5$, in the asymptotic sense.

\begin{Corollary}\label{cor:third_max_inter}
Let $k>d+1\geq 4$. For any $n>n_1(k,d)=d+2(k-d)^2(k^{k-d}-1)^kk!$, if $\mathcal F\subseteq \binom{[n]}{k}$ is the third largest maximal non-trivial $d$-wise intersecting family, then the following hold.
\begin{enumerate}
\item[$(1)$] If $k\leq 2d$, then $\mathcal F\cong\mathcal H(k,d,k-d+1)$.
\item[$(2)$] If $k=2d+1$, then $\mathcal F\cong\mathcal H(2d+1,d,2)$.
\item[$(3)$] If $k>2d+1$, then $\mathcal F\cong\mathcal G(k,d)$.
\end{enumerate}
\end{Corollary}

\begin{Corollary}\label{cor:fourth_max_inter}
Let $k>d+1\geq 4$. For any $n>n_1(k,d)=d+2(k-d)^2(k^{k-d}-1)^kk!$, if $\mathcal F\subseteq \binom{[n]}{k}$ is the fourth largest maximal non-trivial $d$-wise intersecting family, then the following hold.
\begin{enumerate}
\item[$(1)$] If $k=d+2$, then $\mathcal F\cong\mathcal G(d+2,d)$.
\item[$(2)$] If $k=d+3$, then $\mathcal F\cong\mathcal H(d+3,d,3)$.
\item[$(3)$] If $d+4\leq k\leq 2d+1$, then $\mathcal F\cong\mathcal G(k,d)$.
\item[$(4)$] If $k>2d+1$ and $d=3$, then $\mathcal F\cong\mathcal S_1(k,3)$.
\item[$(5)$] If $k>2d+1$ and $d\geq4$, then $\mathcal F\cong\mathcal H(k,d,k-d)$.
\end{enumerate}
\end{Corollary}

\begin{Corollary}\label{cor:fifth_max_inter_d=3}
Let $k\geq 5$. For any $n>n_1(k,3)=3+2(k-3)^2(k^{k-3}-1)^kk!$, if $\mathcal F\subseteq \binom{[n]}{k}$ is the fifth largest maximal non-trivial $3$-wise intersecting family, then the following hold.
\begin{enumerate}
\item[$(1)$] If $k=5$, then $\mathcal F\cong\mathcal S_1(5,3)$.
\item[$(2)$] If $k=6$, then $\mathcal F\cong\mathcal G(6,3)$.
\item[$(3)$] If $k=7$, then $\mathcal F\cong\mathcal H(7,3,4)$.
\item[$(4)$] If $k\geq8$, then $\mathcal F\cong\mathcal S(k,3)$.
\end{enumerate}
\end{Corollary}

\begin{Corollary}\label{cor:sixth_max_inter_d=3}
Let $k\geq5$. For any $n>n_1(k,3)=3+2(k-3)^2(k^{k-3}-1)^kk!$, if $\mathcal F\subseteq \binom{[n]}{k}$ is the sixth largest maximal non-trivial $3$-wise intersecting family, then the following hold.
\begin{enumerate}
\item[$(1)$] If $k=5$, then $\mathcal F\cong\mathcal S(5,3)$.
\item[$(2)$] If $k\in\{6,7\}$, then $\mathcal F\cong\mathcal S_1(k,3)$.
\item[$(3)$] If $k=8$, then $\mathcal F\cong\mathcal H(8,3,5)$.
\item[$(4)$] If $k\geq9$, then $\mathcal F\cong\mathcal S_2(k,3)$ or $\mathcal F\cong\mathcal S_3(k,3)$.
\end{enumerate}
\end{Corollary}

\section{Proof of Theorem \ref{thm:maximal_intersecting_d+2}}\label{sec:2}

We open this section with a basic property about non-trivial $d$-wise intersecting families.


\begin{Lemma}\label{lem:inter_size}{\rm \cite[Lemma 2]{OV}}
Let $\mathcal F\subseteq\binom{[n]}{k}$ be a non-trivial $d$-wise intersecting family. For every $2\leq m\leq d$ and $A_1,\dots,A_m\in\mathcal F$, $|\bigcap_{i=1}^{m}A_i|\geq d-m+1$.
\end{Lemma}

We say that $\mathcal G \subseteq 2^{[n]}$ is a {\em $t$-intersecting family}, if $|G_1\cap G_2|\geq t$ for any $G_1,G_2\in \mathcal G$.
Lemma \ref{lem:inter_size} implies that if $\mathcal F\subseteq\binom{[n]}{k}$ is a non-trivial $d$-wise intersecting family then $\mathcal F$ is a $(d-1)$-intersecting family. For $t$-intersecting families, there is the following folklore result.

\begin{Lemma}\label{lem:(t+1)_unif_t_inter}
If $\mathcal F\subseteq\binom{[n]}{t+1}$ is a $t$-intersecting family, then $\mathcal F$ is a subfamily of either  $\binom{[t+2]}{t+1}$ or $\{[t]\cup\{i\}:i\in[t+1,n]\}$ up to isomorphism.
\end{Lemma}

\subsection{$\Delta$-systems}

We will prove Theorem \ref{thm:maximal_intersecting_d+2} by using the $\Delta$-system method that was initially studied in \cite{DEF,ER,ER60,F78,F80}. A {\em $\Delta$-system} (or {\em sunflower}) with the {\em kernel} $X\subseteq [n]$ is a set ${\cal D}=\{D_1,D_2,\dots,D_s\}\subseteq 2^{[n]}$ such that $D_i\cap D_j=X$ for every $1\leq i< j\leq s$. For $\mathcal F\subseteq 2^{[n]}$ and $X\subseteq [n]$, define the {\em kernel degree} of $X$ in $\cal F$, $\delta_{\mathcal F}(X)$, to be the size of the largest $\Delta$-system in $\mathcal F$ admitting $X$ as the kernel.

The following lemma, discovered by Erd\H{o}s and Rado \cite{ER}, is a fundamental result in extremal set theory. It asserts that a large enough uniform family must contain a $\Delta$-system, regardless of the size of the $\Delta$-system.



\begin{Lemma}[Sunflower Lemma]\label{lem:Sunflower}{\rm \cite{ER}}
If $\mathcal F\subseteq\binom{[n]}{k}$ with $|\mathcal F|>k!(s-1)^k$, then $\mathcal F$ contains a $\Delta$-system of size $s$.
\end{Lemma}

Lemma \ref{lem:Sunflower} does not provide information on the size of the kernel of a $\Delta$-system in $\cal F$. Let $\mathcal F\subseteq\binom{[n]}{k}$ be a non-trivial $d$-wise intersecting family. Lemma \ref{lem:inter_size} shows that $\mathcal F$ is $(d-1)$-intersecting, and hence for any $X\in \bigcup_{i=0}^{d-2}\binom{[n]}{i}$, there is no $\Delta$-system of size $2$ in $\mathcal F$ that has the kernel $X$. For $X\in\binom{[n]}{d-1}$, the following lemma estimates the size of a $\Delta$-system in $\mathcal F$ admitting $X$ as the kernel. It is a slight improvement of \cite[Lemma 3]{OV} and is mentioned in the concluding section of \cite{OV} without proof.


\begin{Lemma}\label{lem:center_size}
Let $k\geq d\geq 3$. If $\mathcal F\subseteq\binom{[n]}{k}$ is a non-trivial $d$-wise intersecting family and $X\in\binom{[n]}{d-1}$, then $\delta_{\mathcal F}(X)<k-d+3$.
\end{Lemma}

\begin{proof} Without loss of generality, suppose, to the contrary, that $\{F_1,F_2,\dots,F_{k-d+3}\}\subseteq \mathcal F$ is a $\Delta$-system of size $k-d+3$ with $X=[d-1]$ as the kernel. By the non-triviality of $\mathcal F$, for each $i\in[d-1]$, there exists $X_{i}\in \mathcal F$ so that $i\notin X_i$.
Since $|F_1\cap F_2|=d-1$ and $\mathcal F$ is $d$-wise intersecting, for any $1\leq j \leq d-1$, $$F_1\cap F_2\cap(\bigcap_{i\in[d-1]\setminus\{j\}}X_i)=\{j\},$$
which implies that $X_i\cap [d-1]=[d-1]\setminus\{i\}$ for each $i\in[d-1]$.
On the other hand, $|X_i\cap F_l|\geq d-1$ for any $l\in[k-d+3]$ by Lemma \ref{lem:inter_size}, and hence $|X_i\cap(F_l\setminus[d-1])|\geq 1$. Therefore,
$$|X_i|\geq |X_i\cap[d-1]|+\sum_{l=1}^{k-d+3}|X_i\cap(F_l\setminus[d-1])|>k,$$
a contradiction.
\end{proof}

Let $\mathcal F\subseteq\binom{[n]}{k}$. Lemma \ref{lem:center_size} shows that for every $(d-1)$-subset $X\subseteq [n]$, the kernel degree of $X$ in $\mathcal F$ is less than $k-d+3$. We are interested in kernels that have larger degrees since they intersect more elements in $\cal F$ (see Lemma \ref{lem:bas_inter}). Let $\mathcal B_1(\mathcal F)$ be a collection of subsets of $[n]$ such that (1) for each $X\in \mathcal B_1(\mathcal F)$, $d\leq |X|<k$, and (2) there exists a $\Delta$-system $\mathcal D\subseteq\binom{[n]}{k}$ of cardinality $k^{|X|-d+1}$ in $\cal F$ that admits $X$ as its kernel. That is, $\mathcal B_1(\mathcal F)$ consists of kernels of $\Delta$-systems with large sizes.

To prove Lemmas \ref{lem:lem5}, \ref{lem:bas_inter} and \ref{lem:bas_sunflower}, we will often implicitly use a simple fact: given $X,Y\subseteq[n]$ and $\mathcal G\subseteq \binom{[n]}{k}$, if $\delta_{\mathcal G}(X)>|Y|$ then there exists $F\in\mathcal G$ such that $X\subseteq F$ and $(F\setminus X)\cap Y=\emptyset$.

\begin{Lemma}\label{lem:lem5}
Let $k>d\geq3$ and $\mathcal F\subseteq \binom{[n]}{k}$ be a maximal non-trivial $d$-wise intersecting family. If $B\in\mathcal B_1(\mathcal F)$, then $\{F\in\binom{[n]}{k}:B\subseteq F\}\subseteq \mathcal F$.
\end{Lemma}

\begin{proof}
Suppose, to the contrary, that there exist $B\in\mathcal B_1(\mathcal F)$ and $F\in\binom{[n]}{k}\setminus\mathcal F$ such that $B\subseteq F$.
Since $\mathcal F$ is maximal, there exist $F_1,\dots,F_{d-1}\in\mathcal F$ such that $F\cap(\bigcap_{i\in[d-1]}F_i)=\emptyset$, and so $B\cap(\bigcap_{i\in[d-1]}F_i)=\emptyset.$
Due to $k^{|B|-d+1}\geq k$, by the definition of $\mathcal B_1(\mathcal F)$, there exists a $\Delta$-system $\mathcal D$ of cardinality $k$ with $B$ as the kernel in $\mathcal F$.
Since $|\bigcap_{i\in[d-1]}F_i|<k=|\mathcal D|$, there exists $F'\in\mathcal D\subseteq\mathcal F$ such that $(F'\setminus B)\cap(\bigcap_{i\in[d-1]}F_i)=\emptyset$, and so $F'\cap(\bigcap_{i\in[d-1]}F_i)=\emptyset$, a contradiction because of the $d$-wise intersecting property of $\mathcal F$.
\end{proof}

Lemma \ref{lem:lem5} motivates us to study the set of minimal members in $\mathcal B_1(\mathcal F)$. Let $\mathcal B_2(\mathcal F)$ be the set of minimal members in $\mathcal B_1(\mathcal F)$ in the usual set inclusion sense, i.e.
$$\mathcal B_2(\mathcal F)=\{X\in\mathcal B_1(\mathcal F): \text{ there is no } Y\in\mathcal B_1(\mathcal F) \text{ such that }Y\subseteq X\}.$$
Let
$$\mathcal B_3(\mathcal F)=\{F\in\mathcal F: \text{ there is no } X\in\mathcal B_2(\mathcal F) \text{ such that }X\subseteq F\}.$$
Write $\mathcal B(\mathcal F)=\mathcal B_2(\mathcal F)\cup\mathcal B_3(\mathcal F)$. For $d\leq i\leq k$, let
$$\mathcal B^{(i)}=\{T\in\mathcal B(\mathcal F):|T|=i\}.$$
It follows that $\mathcal B(\mathcal F)=\bigcup_{i=d}^{k}\mathcal B^{(i)}$ and for each $F\in\mathcal F$, there exists a set $B\in\mathcal B(\mathcal F)$ such that $B\subseteq F$.

The following lemma shows that $\mathcal B(\mathcal F)$ is $(d-1)$-intersecting. It is a generalization of \cite[Lemmas 4 and 5]{OV}.

\begin{Lemma}\label{lem:bas_inter}
Let $k>d\geq 3$. If $\mathcal F\subseteq\binom{[n]}{k}$ is a non-trivial $d$-wise intersecting family, then $\mathcal B(\mathcal F)$ is a $(d-1)$-intersecting family.
Moreover, $|B\cap F|\geq d-1$ for any $B\in\mathcal B(\mathcal F)$ and $F\in\mathcal F$.
\end{Lemma}

\begin{proof} If $B_1,B_2\in \mathcal B_3(\mathcal F)$, Lemma \ref{lem:inter_size} implies that $|B_1\cap B_2|\geq d-1$.
If $B_1\in \mathcal B_2(\mathcal F)$ and $B_2\in \mathcal B_3(\mathcal F)$, suppose, to the contrary,
that $|B_1\cap B_2|<d-1$.
By the definition of $\mathcal B_2(\mathcal F)$, there exists a $\Delta$-system $\{F_{1},F_{2},\dots,F_{k^{|B_1|-d+1}}\}$ with the kernel $B_1$ in $\mathcal F$.
By Lemma \ref{lem:inter_size}, for each $1\leq j\leq k^{|B_1|-d+1}$,
$|(F_j\setminus B_1)\cap B_2|=|F_j\cap B_2|-|B_1\cap B_2|\geq d-1-|B_1\cap B_2|$. Thus
\begin{equation}
\begin{aligned}
|B_2|\geq |B_1\cap B_2|+\sum_{j=1}^{k^{|B_1|-d+1}}|(F_j\setminus B_1)\cap B_2|\geq |B_1\cap B_2|+k(d-1-|B_1\cap B_2|)>k,\notag
\end{aligned}
\end{equation}
a contradiction. If $B_1,B_2\in \mathcal B_2(\mathcal F)$, assume that $|B_1\cap B_2|<d-1$. For $i\in\{1,2\}$, there exists a $\Delta$-system ${\cal D}_i$ in $\mathcal F$ with the kernel $B_i$. Due to $k^{|B_1|-d+1}>|B_2|$, there exists $D_1\in {\cal D}_1$ such that $(D_1\setminus B_1)\cap B_2=\emptyset$. It follows that for every $D_2\in {\cal D}_2$, $|D_1\cap D_2|=|D_1\cap(D_2\setminus B_2)|+|D_1\cap B_2|=|D_1\cap(D_2\setminus B_2)|+|B_1\cap B_2|$, which implies $|D_1\cap(D_2\setminus B_2)|\geq d-1-|B_1\cap B_2|$ by Lemma \ref{lem:inter_size}. Thus
$$|D_1|\geq |B_1\cap B_2|+\sum_{D_2\in {\cal D}_2}|D_1\cap(D_2\setminus B_2)|\geq |B_1\cap B_2|+k(d-1-|B_1\cap B_2|)>k,$$
a contradiction. Therefore, $\mathcal B(\mathcal F)$ is a $(d-1)$-intersecting family.

Furthermore, for each $F\in\mathcal F$, there exists $B_0\in\mathcal B(\mathcal F)$ such that $B_0\subseteq F$.
Since $\mathcal B(\mathcal F)$ is $(d-1)$-intersecting, $|B\cap B_0|\geq d-1$ for any $B\in\mathcal B(\mathcal F)$, and so $|B\cap F|\geq d-1$ for any $B\in\mathcal B(\mathcal F)$ and $F\in\mathcal F$.
\end{proof}

\begin{Lemma}\label{lem:bas_sunflower}
For $4\leq d+1\leq i\leq k$, there is no $\Delta$-system of size $k^{i-d}$ in $\mathcal B^{(i)}$.
\end{Lemma}

\begin{proof} Suppose for contradiction that $\mathcal D=\{B_1,B_2,\dots,B_{k^{i-d}}\}$ is a $\Delta$-system in $\mathcal B^{(i)}$ with the kernel $B$. Then $|B|\leq i-1$ because of $k^{i-d}>1$. By Lemma \ref{lem:bas_inter}, any two members of $\mathcal B^{(i)}$ contain at least $d-1$ common elements, so $|B|\geq d-1$. If $|B|=d-1$, then $\delta_{\cal F}(B)<k-d+3$ by Lemma \ref{lem:center_size}, and so $|\mathcal D|<k-d+3$, which contradicts the fact that $k^{i-d}\geq k-d+3$. Thus $|B|\geq d$. On the whole, $d\leq |B|\leq i-1$.

If $i=k$, then $B\in \mathcal B_1(\mathcal F)$, contradicting the definition of $\mathcal B^{(k)}=\mathcal B_3(\mathcal F)$. Therefore, there is no $\Delta$-system of size $k^{k-d}$ in $\mathcal B^{(k)}$.

Assume that $d+1\leq i\leq k-1$. By the definition of $\mathcal B^{(i)}$, for every $1\leq j\leq k^{i-d}$, there exists a $\Delta$-system ${\cal D}_j$ of size $k^{i-d+1}$ in $\mathcal F$ with the kernel $B_j$.
Since $|\bigcup_{j=2}^{k^{i-d}}B_{j}|<k\cdot k^{i-d}=|{\cal D}_1|$, there exists $D_1\in{\cal D}_1$ such that $$(D_1\setminus B_1)\cap(B_2\cup\dots\cup B_{k^{i-d}})=\emptyset.$$
Since $|D_1\cup(\bigcup_{j=3}^{k^{i-d}}B_{j}|)<k\cdot k^{i-d}=|{\cal D}_2|$, there exists $D_2\in{\cal D}_2$ such that
$$(D_2\setminus B_2)\cap(D_1\cup B_3\cup\dots\cup B_{k^{i-d}})=\emptyset.$$
Similarly, since $|D_1\cup D_2\cup(\bigcup_{j=4}^{k^{i-d}}B_{j})|<k\cdot k^{i-d}=|{\cal D}_3|$, there exists $D_3\in{\cal D}_3$ such that
$$(D_3\setminus B_3)\cap(D_1\cup D_2\cup B_4\cup\dots\cup B_{k^{i-d}})=\emptyset.$$
Continuing the above process, we obtain $D_1,D_2,\dots,D_{k^{i-d}}$ which form a $\Delta$-system in $\cal F$ with the kernel $B$. Since $d\leq |B|\leq i-1$, $B\in \mathcal B_1(\mathcal F)$. This contradicts the fact that $\mathcal B^{(i)}\subseteq \mathcal B_2(\mathcal F)$ and $\mathcal B_2(\mathcal F)$ is the set of minimal members in $\mathcal B_1(\mathcal F)$. Therefore, there is no $\Delta$-system of size $k^{i-d}$ in $\mathcal B^{(i)}$ for $d+1\leq i\leq k-1$.
\end{proof}

\begin{Lemma}\label{lem:size_B^d}
Let $4\leq d+1\leq k$ and $n>d+2(k-d)^2(k^{k-d}-1)^kk!$.
If $\mathcal F\subseteq \binom{[n]}{k}$ is a non-trivial $d$-wise intersecting family and $|\mathcal F|>(k-d-\frac{1}{2})\binom{n-d}{k-d}$, then $|{\cal B}^{(d)}|\geq k-d$.
\end{Lemma}

\begin{proof} Combining Lemmas \ref{lem:Sunflower} and \ref{lem:bas_sunflower}, we have $|\mathcal B^{(i)}|\leq i!(k^{i-d}-1)^i$ for every $d+1\leq i\leq k$. On one hand,
$$|\mathcal F|\leq \sum_{i=d}^{k}|\mathcal B^{(i)}|\binom{n-i}{k-i}\leq |\mathcal B^{(d)}|\binom{n-d}{k-d}+
C\binom{n-d-1}{k-d-1},$$
where $C=\sum_{i=d+1}^{k}i!(k^{i-d}-1)^i\leq (k-d)(k^{k-d}-1)^kk!$ is an absolute constant. On the other hand, by the assumption,
$$|\mathcal F|>(k-d-\frac{1}{2})\binom{n-d}{k-d}.$$
It follows that
\begin{equation}\label{equ:size_Bd}
\begin{aligned}
|\mathcal B^{(d)}|\geq \left\lceil k-d-\frac{1}{2}-C\frac{\binom{n-d-1}{k-d-1}}{\binom{n-d}{k-d}} \right\rceil = \left\lceil k-d-\frac{1}{2}-C\frac{(k-d)}{n-d} \right\rceil \geq k-d,\notag
\end{aligned}
\end{equation}
where the last inequality holds because of $n>d+2(k-d)^2(k^{k-d}-1)^kk!$.
\end{proof}

\subsection{Structure of maximal non-trivial $d$-wise intersecting families}

Let $H$ be a hypergraph. A set $T$ is called a {\em cover} of $H$ if it intersects every hyperedge of $H$.
Let $\nu(H)$ and $\tau(H)$ denote the maximum number of disjoint hyperedges in $H$ and the minimum cardinality of covers of $H$, respectively.

\begin{proof}[{\bf Proof of Theorem \ref{thm:maximal_intersecting_d+2}:}]
Suppose that $\mathcal F\subseteq \binom{[n]}{k}$ is a maximal non-trivial $d$-wise intersecting family and $|\mathcal F|>(k-d-\frac{1}{2})\binom{n-d}{k-d}$. We will prove that $\mathcal F$ must be one of the families $\mathcal H(k,d,l)$, $\mathcal G(k,d)$, ${\cal S}(k,3)$, ${\cal S}_1(k,3)$, ${\cal S}_2(k,3)$ and ${\cal S}_3(k,3)$ up to isomorphism. Note that Appendix \ref{Sec:B} gives lower bounds for the sizes of these extremal families and we can see that their sizes are all greater than $(k-d-\frac{1}{2})\binom{n-d}{k-d}$ for any $n>n_1(k,d)$.

Since $\mathcal B^{(d)}$ is $(d-1)$-intersecting by Lemma \ref{lem:bas_inter}, we can apply Lemma \ref{lem:(t+1)_unif_t_inter} together with renaming elements in $[n]$ such that $\mathcal B^{(d)}$ is a subfamily of $\binom{[d+1]}{d}$ or $\{[d-1]\cup\{i\}:i\in[d,n]\}$. Also by Lemma \ref{lem:size_B^d}, when $n>n_1(k,d)$ and $k>d+1\geq 4$,
\begin{align}\label{eqn:B^d}
|{\cal B}^{(d)}|\geq k-d\geq 2,
\end{align}

\noindent{\underline{\bf Case 1 $\mathcal B^{(d)}\subseteq \binom{[d+1]}{d}$.}}
If $\mathcal B^{(d)}=\{B_1,B_2\}$, then Lemma \ref{lem:bas_inter} implies that $|B_1\cap B_2|=d-1$, and hence $\mathcal B^{(d)}$ is isomorphic to a subfamily of $\{[d-1]\cup\{i\}:i\in[d,n]\}$, which will be discussed in Case 2.

Assume that $|\mathcal B^{(d)}|\geq 3$. Let $B_1,B_2,B_3\in \mathcal B^{(d)}$.
Since $B_j \in \binom{[d+1]}{d}$ for $1\leq j\leq 3$, $B_1\cap B_2$, $B_1\cap B_3$ and $B_2\cap B_3$ are pairwise distinct. If there would be $F\in\mathcal F$ such that $F\cap B_1=F\cap B_2=F\cap B_3$, then $F\cap B_1\subseteq B_1\cap B_2$ and $F\cap B_1\subseteq B_1\cap B_3$.
By Lemma \ref{lem:bas_inter}, $|F\cap B_1|\geq d-1$ and $|B_1\cap B_2|=|B_1\cap B_3|=d-1$, so $F\cap B_1=B_1\cap B_2$ and $F\cap B_1=B_1\cap B_3$, which yields $B_1\cap B_2=B_1\cap B_3$, a contradiction. Thus for any $F\in\mathcal F$, there exist $1\leq j_1<j_2\leq 3$ such that $F\cap B_{j_1}\neq F\cap B_{j_2}$. Since $B_{j_1}\cup B_{j_2}=[d+1]$, for any $F\in\mathcal F$, we have $|F\cap[d+1]|\geq d$, which gives $\mathcal F\subseteq \mathcal H(k,d,2)=\mathcal A(k,d)$. Note that by Lemma \ref{lem:size_H(k,d,l)_and_bound} in Appendix \ref{Sec:B}, $|{\mathcal H}(k,d,2)|>(k-d-\frac{1}{2})\binom{n-d}{k-d}$ only if $k<2d+2$, and only $(1)$, $(2)$ and $(4)$ of Theorem \ref{thm:maximal_intersecting_d+2} deal with the case of $k<2d+2$.

\noindent{\underline{\bf Case 2 $\mathcal B^{(d)}\subseteq\{[d-1]\cup\{i\}:i\in[d,n]\}$.}}
Without loss of generality assume that
$$\mathcal B^{(d)}=\{[d-1]\cup\{i\}:i\in[d,d+q]\}$$
for some $1\leq q\leq n-d$. By \eqref{eqn:B^d}, $q+1=|\mathcal B^{(d)}|\geq k-d$, so $q\geq k-d-1$.
We claim that $q\leq k-d+1$. Suppose for contradiction that $q\geq k-d+2$. By the non-triviality of $\mathcal F$, there exists $X\in\mathcal F$ such that $1\notin X$.
For any $B\in \mathcal B^{(d)}$, $|B\cap X|\geq d-1$ by Lemma \ref{lem:bas_inter}, which implies $B\setminus\{1\}\subseteq X$. Thus $\bigcup_{B\in\mathcal B^{(d)}}(B\setminus\{1\})=[2,d+q]\subseteq X$. This yields a contradiction since $|X|\geq d+q-1\geq k+1$.
Therefore, $$k-d-1\leq q\leq k-d+1.$$

There are three cases to consider according to the values of $q$. Before proceeding to the three cases, we define some subfamilies of $\mathcal F$. Let $$\mathcal F_1=\{F\in\mathcal F:[d-1]\subseteq F,F\cap[d,d+q]\neq\emptyset\},$$
$$\mathcal F_2=\{F\in\mathcal F: |F\cap [d-1]|\leq d-2\},$$
and
$$\mathcal F_3=\{F\in\mathcal F:[d-1]\subseteq F,F\cap[d,d+q]=\emptyset\}.$$
Then $\mathcal F=\mathcal F_1\cup\mathcal F_2\cup\mathcal F_3$, and $\mathcal F_1$, $\mathcal F_2$ and $\mathcal F_3$ are pairwise disjoint. For $i\in[d-1]$, let
$\mathcal F_2^{(i)}=\{F\in\mathcal F:i\notin F\}.$
Then $\mathcal F_2=\bigcup_{i\in[d-1]}\mathcal F_{2}^{(i)}$.
Since $\mathcal F$ is non-trivial, for every $i\in[d-1]$, there exists $X\in\mathcal F$ such that $i\notin X$, and hence $\mathcal F_2^{(i)}\neq\emptyset$ for every $i\in[d-1]$. For any $F\in\mathcal F_2^{(i)}$ and $B\in\mathcal B^{(d)}$, by Lemma \ref{lem:bas_inter},
\begin{equation}\label{equ:size_F2_with_[d-1]}
\begin{aligned}
|F\cap [d-1]|=|F\cap B|-|F\cap(B\setminus[d-1])|\geq d-1-|F\cap(B\setminus[d-1])|\geq d-2.
\end{aligned}
\end{equation}
On the other hand, $|F\cap[d-1]|\leq d-2$ for any $F\in\mathcal F_2^{(i)}$. Thus $|F\cap[d-1]|=d-2$ and by (\ref{equ:size_F2_with_[d-1]}) it holds that $|F\cap(B\setminus[d-1])|=1$. The former yields $F\cap [d-1]=[d-1]\setminus\{i\}$ for any $F\in\mathcal F_2^{(i)}$, and the latter implies $B\setminus[d-1]\subseteq F$ for any $B\in\mathcal B^{(d)}$ and so $[d,d+q]\subseteq F$.
Therefore, for every $i\in[d-1]$,
$$\mathcal F_{2}^{(i)}=\{F\in\mathcal F: i\notin F, [d+q]\setminus\{i\}\subseteq F\}.$$

\noindent{\underline{\bf Case 2.1 $q=k-d+1$.}} For every $i\in[d-1]$,
$\mathcal F_2^{(i)}=\{X_i\}$, where $X_i=[k+1]\setminus\{i\}$.
Thus
$$\mathcal F_2=\{X_i:i\in[d-1]\}.$$
If $\mathcal F_3\neq \emptyset$, then there would be $F_3\in\mathcal F_3$ such that $F_3\cap(\bigcap_{i\in[d-1]}X_i)=\emptyset$, which leads to a contradiction since $\mathcal F$ is $d$-wise intersecting, and so $\mathcal F_3=\emptyset$. It is readily checked that $\mathcal F=\mathcal F_1\cup\mathcal F_2 \subseteq\mathcal H(k,d,k-d+2)$.

\noindent{\underline{\bf Case 2.2 $q=k-d$.}} For every $i\in[d-1]$,
$\mathcal F_{2}^{(i)}=\{F\in\mathcal F: i\notin F,[k]\setminus\{i\}\subseteq F\}$. Thus
$$\mathcal F_2=\{F\in\mathcal F:|F\cap[d-1]|=d-2,[d,k]\subseteq F\}.$$
If $\mathcal F_3=\emptyset$, then $\mathcal F=\mathcal F_1\cup\mathcal F_2 \subseteq\mathcal H(k,d,k-d+1)$. Assume that $\mathcal F_3\neq\emptyset$. If there would be $1\leq j_1<j_2\leq d-1$ such that $F_{j_1}\in\mathcal F_2^{(j_1)}$, $F_{j_2}\in \mathcal F_2^{(j_2)}$ and $F_{j_1}\setminus[k]\neq F_{j_2}\setminus[k]$, then $F_{j_1}\cap F_{j_2}=[k]\setminus\{j_1,j_2\}$. For any $F_3\in\mathcal F_3$, $F_3\cap F_{j_1}\cap F_{j_2}=[d-1]\setminus\{j_1,j_2\}$. Take $X_i\in\mathcal F_2^{(i)}$ for  $i\in[d-1]\setminus\{j_1,j_2\}$.
Since $i\notin X_i$, $F_3\cap F_{j_1}\cap F_{j_2}\cap (\bigcap_{i\in[d-1]\setminus\{j_1,j_2\}}X_i)=\emptyset$, a contradiction.
Thus $F_{j_1}\setminus[k]=F_{j_2}\setminus[k]$ for any $F_{j_1}\in\mathcal F_2^{(j_1)}$ and $F_{j_2}\in \mathcal F_2^{(j_2)}$. This implies that there exists a unique $x\in[k+1,n]$ such that $\{x\}=F\setminus[k]$ for any $F\in{\cal F}_2=\bigcup_{i\in[d-1]}\mathcal F_2^{(i)}$.
Without loss of generality assume that $x=k+1$. Then
$$\mathcal F_2=\{[k+1]\setminus\{i\}:i\in[d-1]\}.$$
Note that $[2,k+1]\in\mathcal F_2$.
For any $F_3\in\mathcal F_3$, since $|[2,k+1]\cap F_3|\geq d-1$ by Lemma \ref{lem:inter_size} and $F_3\cap [d,k]=\emptyset$ by the definition of ${\cal F}_3$, $F_3$ must contain $k+1$.
Therefore,
$${\cal F}_1\cup{\cal F}_3=\{F\in {\cal F}:[d-1]\subseteq F, F\cap[d,k+1]\neq \emptyset\}.$$
It is readily checked that $\mathcal F=\mathcal F_1\cup\mathcal F_2\cup \mathcal F_3\subseteq\mathcal H(k,d,k-d+2)$.

\noindent{\bf \underline{Case 2.3 $q=k-d-1$.}} For every $i\in[d-1]$,
$\mathcal F_{2}^{(i)}=\{F\in\mathcal F:i\notin F,[k-1]\setminus\{i\}\subseteq F\}$. Thus
\begin{align}\label{eqn:f2-611}
\mathcal F_2=\{F\in\mathcal F:|F\cap[d-1]|=d-2,[d,k-1]\subseteq F\}.
\end{align}
If $\mathcal F_3=\emptyset$, then $\mathcal F=\mathcal F_1\cup\mathcal F_2 \subseteq\mathcal H(k,d,k-d)$.

In what follows, we always assume that $\mathcal F_3\neq\emptyset$. We construct an auxiliary graph $G=(V(G),E(G))$, where $V(G)=[k,n]$ and
$$E(G)=\left\{e\in\binom{[k,n]}{2}: e=F\setminus[k-1] ,F\in \mathcal F_2 \right\}.$$
For each $i\in[d-1]$, let $G_i=(V(G),E(G_i))$ be a subgraph of $G$, where
$$E(G_i)=\{e\in E(G):e=F\setminus[k-1] ,F\in \mathcal F_2^{(i)}\}.$$

\noindent{\bf\underline{Case 2.3.1 $\tau(G)=1$.}} If $|E(G)|=1$, let $E(G)=\{e\}$. Then for any $F\in {\cal F}_2$, $F\setminus [k-1]=e$. By \eqref{eqn:f2-611} we have
$$\mathcal F_2=\{F\in\mathcal F:|F\cap[d-1]|=d-2,[d,k-1]\cup e\subseteq F\}.$$
Take $[2,k-1]\cup e \in \mathcal F_2$. It follows from Lemma \ref{lem:inter_size} that for any $F_3\in \mathcal F_3$, $|F_3\cap([2,k-1]\cup e)|=d-2+|F_3\cap e|\geq d-1$, and so $F_3\cap e\neq\emptyset$ for every $F_3\in \mathcal F_3$. Therefore,
$${\cal F}_1\cup{\cal F}_3=\{F\in {\cal F}:[d-1]\subseteq F, F\cap([d,k-1]\cup e)\neq \emptyset\}.$$
Without loss of generality assume that $e=\{k,k+1\}$. Then $\mathcal F\subseteq\mathcal H(k,d,k-d+2)$.

If $|E(G)|=m\geq2$, due to $\tau(G)=1$, let $E(G)=\{\{y,y_1\},\{y,y_2\},\dots,\{y,y_m\}\}$. By \eqref{eqn:f2-611},
$$\mathcal F_2=\{F\in\mathcal F:|F\cap[d-1]|=d-2,[d,k-1]\cup \{y\} \subseteq F\}.$$
Since $m\geq 2$, there exist $1\leq j_1<j_2\leq d-1$ such that $F_{j_1}=([k-1]\setminus\{j_1\})\cup\{y,y_{m_1}\}\in \mathcal F_2^{(j_1)}$ and $F_{j_2}=([k-1]\setminus\{j_2\})\cup\{y,y_{m_2}\}\in \mathcal F_2^{(j_2)}$, where $1\leq m_1,m_2\leq m$ and $m_1\neq m_2$. So $F_{j_1}\cap F_{j_2}=([k-1]\setminus\{j_1,j_2\})\cup\{y\}$. Take $X_i\in\mathcal F_2^{(i)}$ for $i\in[d-1]\setminus\{j_1,j_2\}$. Then $F_{j_1}\cap F_{j_2}\cap (\bigcap_{i\in[d-1]\setminus\{j_1,j_2\}}X_i)=[d,k-1]\cup\{y\}$. If there would be $F_3\in\mathcal F_3$ such that $y\notin F_3$, then $F_3\cap F_{j_1}\cap F_{j_2}\cap (\bigcap_{i\in[d-1]\setminus\{j_1,j_2\}}X_i)=\emptyset$, a contradiction. Hence $y\in F_3$ for any $F_3\in \mathcal F_3$. This yields
$${\cal F}_1\cup{\cal F}_3=\{F\in {\cal F}:[d-1]\subseteq F, F\cap([d,k-1]\cup \{y\})\neq \emptyset\}.$$
Without loss of generality assume that $y=k$. Then $\mathcal F\subseteq\mathcal H(k,d,k-d+1)$.

\noindent{\bf \underline{Case 2.3.2 $\nu(G)\geq2$.}} Suppose
$$e_1=\{x_1,y_1\} \text{ and } e_2=\{x_2,y_2\}$$
are two disjoint edges in $G$.
If there would be $1\leq j_1< j_2\leq d-1$ such that $F_{j_1}=([k-1]\setminus\{j_1\})\cup e_1\in\mathcal F_2^{(j_1)}$ and $F_{j_2}=([k-1]\setminus\{j_2\})\cup e_2\in\mathcal F_2^{(j_2)}$, then for any $X_i\in\mathcal F_2^{(i)}$ with $i\in[d-1]\setminus\{j_1,j_2\}$ and any $F_3\in\mathcal F_3$, we have
$F_3\cap F_{j_1}\cap F_{j_2}\cap (\bigcap_{i\in[d-1]\setminus\{j_1,j_2\}}X_i)=\emptyset,$
a contradiction. Thus $e_1$ and $e_2$ must lie in the same graph $G_j$ for some unique $j\in[d-1]$. Without loss of generality, assume that $j=1$, and write
$$F_{1,1}=[2,k-1]\cup e_1\in\mathcal F_2^{(1)}\ \ {\rm and}\ \ F_{1,2}=[2,k-1]\cup e_2\in\mathcal F_2^{(1)}.$$

Given $h\in[2,d-1]$, let $F_h=([k-1]\setminus\{h\})\cup e\in \mathcal F_2^{(h)}$ for some $e\in E(G)$.  Since $\mathcal F$ is $d$-wise intersecting, for any $F_3\in\mathcal F_3$ and $X_i\in\mathcal F_2^{(i)}$ with $i\in[2,d-1]\setminus\{h\}$, we have
$F_3\cap F_{1,1}\cap F_h\cap(\bigcap_{i\in[2,d-1]\setminus\{h\}}X_i)=F_3\cap e_1\cap e\subseteq e_1\cap e\neq\emptyset$
and
$F_3\cap F_{1,2}\cap F_h\cap(\bigcap_{i\in[2,d-1]\setminus\{h\}}X_i)=F_3\cap e_2\cap e\subseteq e_2\cap e\neq\emptyset$. Thus $e$ must be one of the following edges:
$$e_3=\{x_1,x_2\}, \text{   } e_4=\{y_1,y_2\},\text{   } e_5=\{x_1,y_2\} \text{ and } e_6=\{x_2,y_1\}.$$
That is, $\bigcup_{i\in[2,d-1]}E(G_i)\subseteq\{e_3,e_4,e_5,e_6\}$.

When $d\geq 4$, let $2\leq j_1<j_2\leq d-1$. Let $F_{j_1}=([k-1]\setminus\{j_1\})\cup e_{h_1}\in\mathcal F_{2}^{(j_1)}$ and $F_{j_2}=([k-1]\setminus\{j_2\})\cup e_{h_2}\in\mathcal F_{2}^{(j_2)}$, where $e_{h_1}$ and $e_{h_2}$ come from $\{e_3,e_4,e_5,e_6\}$. For any $F_3\in\mathcal F_3$ and $X_i\in\mathcal F_2^{(i)}$ with $i\in[2,d-1]\setminus\{j_1,j_2\}$, if $\{e_{h_1},e_{h_2}\}\in\{\{e_3,e_4\},\{e_3,e_5\},\{e_4,e_6\}\}$, then $F_3\cap F_{1,2}\cap F_{j_1}\cap F_{j_2}\cap(\bigcap_{i\in[2,d-1]\setminus\{j_1,j_2\}}X_i)=\emptyset,$
a contradiction; if $\{e_{h_1},e_{h_2}\}\in\{\{e_4,e_5\},\{e_3,e_6\},\{e_5$, $e_6\}\}$, then $F_3\cap F_{1,1}\cap F_{j_1}\cap F_{j_2}\cap(\bigcap_{i\in[2,d-1]\setminus\{j_1,j_2\}}X_i)=\emptyset,$
a contradiction.
Therefore, when $d\geq 4$, $\bigcup_{i\in[2,d-1]}E(G_i)$ consists of only one edge, and if $|\bigcup_{i\in[2,d-1]}E(G_i)|\geq2$, then $d=3$.

If $|\bigcup_{i\in[2,d-1]}E(G_i)|=1$, without loss of generality assume that $\bigcup_{i\in[2,d-1]}E(G_i)=\{e_3\}.$ That is, for every $i\in[2,d-1]$,
$$\mathcal F_2^{(i)}=\{A_i\}\text{, where }A_i=([k-1]\setminus\{i\})\cup e_3.$$
For any $F_3\in\mathcal F_3$, since $F_3\cap F_{1,1}\cap(\bigcap_{i\in[2,d-1]}A_i)=F_3\cap\{x_1\}\neq\emptyset$ and $F_3\cap F_{1,2}\cap(\bigcap_{i\in[2,d-1]}A_i)=F_3\cap\{x_2\}\neq\emptyset$, we have $\{x_1,x_2\}\subseteq F_3$. Thus
$$\mathcal F_3=\{F\in\mathcal F:[d-1]\subseteq F,F\cap[d,k-1]=\emptyset,e_3=\{x_1,x_2\}\subseteq F\}.$$
For any $F_1\in\mathcal F_2^{(1)}$ and $F_3\in\mathcal F_3$, since $F_3\cap F_1\cap(\bigcap_{i\in[2,d-1]}A_i)=F_1\cap e_3\neq\emptyset$, by \eqref{eqn:f2-611},
$$\mathcal F_{2}^{(1)}=\{F\in\mathcal F:[2,k-1]\subseteq F, 1\not\in F, F\cap\{x_1,x_2\}\neq \emptyset\}.$$
Without loss of generality take $x_1=k$ and $x_2=k+1$. It is readily checked that $\mathcal F\subseteq\mathcal G(k,d)$.

If $|\bigcup_{i\in[2,d-1]}E(G_i)|\geq 2$, then $d=3$, and so $|E(G_2)|\geq 2$. There are four cases to consider according to the size of $E(G_2)$.

If $|E(G_2)|=2$ and the two edges in $G_2$ are disjoint, then without loss of generality assume that $E(G_2)=\{e_3,e_4\}$. This implies that
$$\mathcal F_2^{(2)}=\{F_{2,1},F_{2,2}\}\text{, where }F_{2,1}=[3,k-1]\cup\{1\}\cup e_3\text{ and }F_{2,2}=[3,k-1]\cup\{1\}\cup e_4.$$
For any $F_3\in\mathcal F_3$, since
$F_3\cap F_{1,1}\cap F_{2,1}=F_3\cap\{x_1\}\neq\emptyset$, $F_3\cap F_{1,2}\cap F_{2,1}=F_3\cap\{x_2\}\neq\emptyset,$
$F_3\cap F_{1,1}\cap F_{2,2}=F_3\cap\{y_1\}\neq\emptyset$ and $F_3\cap F_{1,2}\cap F_{2,2}=F_3\cap\{y_2\}\neq\emptyset,$
we have $\{x_1,x_2,y_1,y_2\}\subseteq F_3$ and so
$$\mathcal F_3=\{F\in\mathcal F:[2]\subseteq F,F\cap[3,k-1]=\emptyset, \{x_1,x_2,y_1,y_2\}\subseteq F\},$$
which implies $k\geq 6$. For any $F_1\in\mathcal F_2^{(1)}$ and $F_3\in\mathcal F_3$, $F_3\cap F_1\cap F_{2,1}=F_1\cap e_3\neq\emptyset$ and $F_3\cap F_1\cap F_{2,2}=F_1\cap e_4\neq\emptyset$. Since $e_3$ and $e_4$ are disjoint, $F_1$ must contain at least one of $e_1,e_2,e_5$ and $e_6$.
Therefore,
$$\mathcal F_{2}^{(1)}\subseteq \{[2,k-1]\cup e : e\in\{e_1,e_2,e_5,e_6\}\}.$$
Without loss of generality assume that $x_1=k$, $y_1=k+1$, $x_2=k+2$ and $y_2=k+3$. It is readily checked that $\mathcal F\subseteq\mathcal S_2(k,3)$. Note that only $(2)$ and $(3)$ of Theorem \ref{thm:maximal_intersecting_d+2} deal with the case of $d=3$ and $k\geq 6$.

If $|E(G_2)|=2$ and the two edges in $G_2$ have a common element, then without loss of generality assume that $E(G_2)=\{e_3,e_6\}$. This implies that
$$\mathcal F_2^{(2)}=\{F_{2,1},F_{2,2}\},\text{ where }F_{2,1}=[3,k-1]\cup\{1\}\cup e_3\text{ and }F_{2,2}=[3,k-1]\cup\{1\}\cup e_6.$$
For any $F_3\in\mathcal F_3$, since
$F_3\cap F_{1,1}\cap F_{2,1}=F_3\cap\{x_1\}\neq\emptyset,$
$F_3\cap F_{1,1}\cap F_{2,2}=F_3\cap\{y_1\}\neq\emptyset$
and
$F_3\cap F_{1,2}\cap F_{2,2}=F_3\cap\{x_2\}\neq\emptyset,$
we have $\{x_1,x_2,y_1\}\subseteq F_3$ and so
$$\mathcal F_3=\{F\in\mathcal F:[2]\subseteq F,F\cap[3,k-1]=\emptyset, \{x_1,x_2,y_1\}\subseteq F\},$$
which implies $k\geq 5$. For any $F_1\in\mathcal F_2^{(1)}$ and $F_3\in\mathcal F_3$, $F_3\cap F_1\cap F_{2,1}=F_1\cap e_3\neq\emptyset$ and $F_3\cap F_1\cap F_{2,2}=F_1\cap e_6\neq\emptyset$. Thus either $e_1=\{x_1,y_1\}\subseteq F_1$ or $x_2\in F_1$.
Therefore,
$$\mathcal F_{2}^{(1)}=\{F_{1,1}\}\cup \{F\in {\cal F}: [2,k-1]\subseteq F, 1\not\in F, x_2\in F\}.$$
Without loss of generality assume that $x_1=k+1$, $y_1=k+2$, $x_2=k$ and $y_2=k+3$. It is readily checked that $\mathcal F\subseteq\mathcal S_1(k,3)$. Note that only $(1)$, $(2)$ and $(3)$ of Theorem \ref{thm:maximal_intersecting_d+2} deal with the case of $d=3$ and $k\geq 5$.

If $|E(G_2)|=3$, then without loss of generality assume that $E(G_2)=\{e_3,e_4,e_6\}$. This implies that
$$\mathcal F_2^{(2)}=\{F_{2,1},F_{2,2},F_{2,3}\},$$
where
$$F_{2,1}=[3,k-1]\cup\{1\}\cup e_3,F_{2,2}=[3,k-1]\cup\{1\}\cup e_4\text{ and }F_{2,3}=[3,k-1]\cup\{1\}\cup e_6.$$
For any $F_3\in\mathcal F_3$, since
$F_3\cap F_{1,1}\cap F_{2,1}=\{x_1\}\cap F_3\neq\emptyset$, $F_3\cap F_{1,1}\cap F_{2,2}=\{y_1\}\cap F_3\neq\emptyset$, $F_3\cap F_{1,2}\cap F_{2,1}=\{x_2\}\cap F_3\neq\emptyset$ and $F_3\cap F_{1,2}\cap F_{2,2}=\{y_2\}\cap F_3\neq\emptyset$, we have $\{x_1,x_2,y_1,y_2\}\subseteq F_3$, and so
$$\mathcal F_3=\{F\in\mathcal F:[2]\subseteq F,F\cap[3,k-1]=\emptyset, \{x_1,x_2,y_1,y_2\}\subseteq F\},$$
which implies $k\geq 6$.
For any $F_1\in\mathcal F_2^{(1)}$ and $F_3\in\mathcal F_3$, since $F_3\cap F_1\cap F_{2,1}=F_1\cap e_3\neq\emptyset$, $F_3\cap F_1\cap F_{2,2}=F_1\cap e_4\neq\emptyset$ and $F_3\cap F_1\cap F_{2,3}=F_1\cap e_6\neq\emptyset$, $F_1$ must contain at least one of $e_1,e_2$ and $e_6$.
Therefore,
$$\mathcal F_{2}^{(1)}\subseteq \{[2,k-1]\cup e : e\in\{e_1,e_2,e_6\}\}.$$
Without loss of generality assume that $x_1=k$, $y_1=k+2$, $x_2=k+1$ and $y_2=k+3$. It is readily checked that $\mathcal F\subseteq \mathcal S_3(k,3)$. Note that only $(2)$ and $(3)$ of Theorem \ref{thm:maximal_intersecting_d+2} deal with the case of $d=3$ and $k\geq 6$.

If $|E(G_2)|=4$, then
\begin{align*}
\mathcal F_2^{(2)}=\{F_{2,1},F_{2,2},F_{2,3},F_{2,4}\},
\end{align*}
where
\begin{align*}
F_{2,1}=[3,k-1]\cup\{1\}\cup e_3, &\text{\ \ \ \ \ \ \ }F_{2,2}=[3,k-1]\cup\{1\}\cup e_4,\\
F_{2,3}=[3,k-1]\cup\{1\}\cup e_5\ &\text{\ \ \ \ \ \ \ }F_{2,4}=[3,k-1]\cup\{1\}\cup e_6.
\end{align*}
For any $F_3\in\mathcal F_3$, since
$F_3\cap F_{1,1}\cap F_{2,1}=\{x_1\}\cap F_3\neq\emptyset$, $F_3\cap F_{1,1}\cap F_{2,2}=\{y_1\}\cap F_3\neq\emptyset$, $F_3\cap F_{1,2}\cap F_{2,1}=\{x_2\}\cap F_3\neq\emptyset$ and $F_3\cap F_{1,2}\cap F_{2,2}=\{y_2\}\cap F_3\neq\emptyset$, we have
$\{x_1,x_2,y_1,y_2\}\subseteq F_3$, and so
$$\mathcal F_3=\{F\in\mathcal F:[2]\subseteq F,F\cap[3,k-1]=\emptyset, \{x_1,x_2,y_1,y_2\}\subseteq F\},$$
which implies $k\geq 6$. For any $F_1\in\mathcal F_2^{(1)}$, since $F_3\cap F_1\cap F_2\neq\emptyset$ for any $F_2\in\mathcal F_2^{(2)}$, $F_1$ must contain at least one of $e_1$ and $e_2$.
Therefore,
$$\mathcal F_{2}^{(1)}=\{F_{1,1},F_{1,2}\}=\{[2,k-1]\cup e : e\in\{e_1,e_2\}\}.$$
Without loss of generality assume that $x_1=k$, $y_1=k+2$, $x_2=k+1$ and $y_2=k+3$. Also we interchange the elements 1 and 2. Then it is readily checked that $\mathcal F\subseteq\mathcal S_2(k,3)$. Note that only $(2)$ and $(3)$ of Theorem \ref{thm:maximal_intersecting_d+2} deal with the case of $d=3$ and $k\geq 6$.

\noindent{\bf \underline{Case 2.3.3 $\nu(G)=1$ and {\bf$\tau(G)\geq 2$.}}} In this case $G$ must be a triangle and $\tau(G)=2$. Without loss of generality assume that $E(G)=\{e_1,e_2,e_3\}$ where
$$e_1=\{k,k+1\},\text{   }e_2=\{k,k+2\},\text{   }e_3=\{k+1,k+2\},$$
and assume that $\mathcal F_{2}^{(1)}$ has the largest size among all $\mathcal F_{2}^{(i)}$ with $i\in[d-1]$, that is, $|\mathcal F_{2}^{(1)}|=\max\{|\mathcal F_{2}^{(i)}|:i\in[d-1]\}$. Clearly $1\leq |\mathcal F_{2}^{(1)}|\leq 3$.

\noindent{\bf \underline{Case 2.3.3.1 $|\mathcal F_{2}^{(1)}|=3$.}} In the case,
$$\mathcal F_{2}^{(1)}=\{F_{1,1},F_{1,2},F_{1,3}\},$$
where
$$F_{1,1}=[2,k-1]\cup e_1,F_{1,2}=[2,k-1]\cup e_2\text{ and }F_{1,3}=[2,k-1]\cup e_3.$$
For any $F_3\in\mathcal F_3$, it follows from Lemma \ref{lem:inter_size} that $F_3\cap e_i\neq\emptyset$ for each $1\leq i\leq 3$, and so $|F_3\cap[k,k+2]|\geq2$.

If there exists $F_3\in\mathcal F_3$ such that $|F_3\cap[k,k+2]|=2$, then without loss of generality, assume that $F_3\cap[k,k+2]=e_1$.
If there would be $j\in[2,d-1]$ such that $F_2^j=([k-1]\setminus\{j\})\cup e_2\in\mathcal F_2^{(j)}$ (resp. $F_2^j=([k-1]\setminus\{j\})\cup e_3\in\mathcal F_2^{(j)}$), then for any $X_i\in\mathcal F_2^{(i)}$ with  $i\in[2,d-1]\setminus\{j\}$, we have $F_3\cap F_{1,3}\cap F_2^j\cap(\bigcap_{i\in[2,d-1]\setminus\{j\}}X_i)=\emptyset$ (resp. $F_3\cap F_{1,2}\cap F_2^j\cap(\bigcap_{i\in[2,d-1]\setminus\{j\}}X_i)=\emptyset$), a contradiction. Thus $\bigcup_{i\in[2,d-1]}E(G_i)$ only contains one edge $e_1$. That is, for every $i\in[2,d-1]$,
$$\mathcal F_2^{(i)}=\{([k-1]\setminus\{i\})\cup e_1\}.$$
The above procedure implies that for any $F'_3\in\mathcal F_3$, we have $F'_3\cap[k,k+2]\neq e_2$ and $F'_3\cap[k,k+2]\neq e_3$. Since $|F'_3\cap[k,k+2]|\geq2$ for any $F'_3\in\mathcal F_3$, we have $e_1\subseteq F'_3$ for any $F'_3\in\mathcal F_3$, and so
$$\mathcal F_3=\{F\in\mathcal F:[d-1]\subseteq F,F\cap[d,k-1]=\emptyset,e_1=\{k,k+1\}\subseteq F\}.$$
It is readily checked that $\mathcal F\subseteq\mathcal G(k,d)$.

If for any $F_3\in\mathcal F_3$, $|F_3\cap[k,k+2]|=3$, i.e., $[k,k+2]\subseteq F_3$, then
$$\mathcal F_3=\{F\in\mathcal F:[d-1]\subseteq F,F\cap[d,k-1]=\emptyset,[k,k+2]\subseteq F\}.$$
When $d\geq4$, if there would be $2\leq j_1<j_2\leq d-1$ such that $F_{2}^{j_1}=([k-1]\setminus\{j_1\})\cup e_l\in\mathcal F_2^{(j_1)}$ and $F_{2}^{j_2}=([k-1]\setminus\{j_2\})\cup e_{l'}\in\mathcal F_2^{(j_2)}$ where $l,l'\in\{1,2,3\}$ and $l\neq l'$, then for $e\in E(G)\setminus\{e_l,e_{l'}\}$ and $X_i\in\mathcal F_2^{(i)}$ with $i\in[2,d-1]\setminus\{j_1,j_2\}$, we have
$F_3\cap ([2,k-1]\cup e)\cap F_{2}^{j_1}\cap F_{2}^{j_2}\cap(\bigcap_{i\in[2,d-1]\setminus\{j_1,j_2\}}X_i)=\emptyset,$
a contradiction. Therefore, when $d\geq 4$, $\bigcup_{i\in[2,d-1]}E(G_i)$ consists of only one edge, and if
$|\bigcup_{i\in[2,d-1]}E(G_i)|\geq2$, then $d=3$.
If $|\bigcup_{i\in[2,d-1]}E(G_i)|=1$, then without loss of generality assume that $\bigcup_{i\in[2,d-1]}E(G_i)=\{e_1\}.$ Thus for every $i\in[2,d-1]$,
$$\mathcal F_2^{(i)}=\{([k-1]\setminus\{i\})\cup e_1\}.$$
It is readily checked that $\mathcal F\subseteq\mathcal G(k,d)$. If $|\bigcup_{i\in[2,d-1]}E(G_i)|\geq2$, then $d=3$. If $|E(G_2)|=2$, then without loss of generality assume that $E(G_2)=\{e_1,e_2\}$, and so
$$\mathcal F_2^{(2)}=\{[3,k-1]\cup\{1\}\cup e_1, [3,k-1]\cup\{1\}\cup e_2\}.$$
It is readily checked that $\mathcal F\subseteq\mathcal S_1(k,3)$. If $|E(G_2)|=3$, then
$$\mathcal F_2^{(2)}=\{[3,k-1]\cup\{1\}\cup e_1, [3,k-1]\cup\{1\}\cup e_2, [3,k-1]\cup\{1\}\cup e_3\}.$$
It is readily checked that $\mathcal F\subseteq\mathcal S(k,3)$.

\noindent{\bf \underline{Case 2.3.3.2 $|\mathcal F_{2}^{(1)}|=2$.}} Without loss of generality, assume that $E(G_1)=\{e_1,e_2\}$ and
$$\mathcal F_{2}^{(1)}=\{F_{1,1}, F_{1,2}\},\text{ where }F_{1,1}=[2,k-1]\cup e_1\text{ and } F_{1,2}=[2,k-1]\cup e_2.$$
There exists $j_1\in[2,d-1]$ such that $e_3\in E(G_{j_1})$. Let $F_{j_1}=([d-1]\setminus\{j_1\})\cup e_3\in\mathcal F_2^{(j_1)}$.
When $d\geq4$, if there would be $j_2\in[2,d-1]\setminus\{j_1\}$ such that $F_{j_2}=([k-1]\setminus\{j_2\})\cup e\in\mathcal F_2^{(j_2)}$ where $e\in E(G_1)=\{e_1,e_2\}$, then for $e'\in E(G_1)\setminus\{e\}$, any $F_3\in\mathcal F_3$ and any $X_i\in\mathcal F_2^{(i)}$ with $i\in[2,d-1]\setminus\{j_1,j_2\}$, we have $F_3\cap ([2,k-1]\cup e')\cap F_{j_1}\cap F_{j_2}\cap(\bigcap_{i\in[2,d-1]\setminus\{j_1,j_2\}}X_i)=\emptyset,$
a contradiction.
Hence $\bigcup_{i\in[2,d-1]\setminus\{j_1\}}E(G_i)$ contains only one edge $e_3$.
Furthermore, if there would be $F'_{j_1}=[2,k-1]\cup e\in\mathcal F_2^{(j_1)}$ where $e\in E(G_1)$, then for $e'\in E(G_1)\setminus\{e\}$, any $F_3\in\mathcal F_3$ and any $X_i\in\mathcal F_2^{(i)}$ with $i\in[2,d-1]\setminus\{j_1\}$, we have
$F_3\cap ([2,k-1]\cup e')\cap F'_{j_1}\cap(\bigcap_{i\in[2,d-1]\setminus\{j_1\}}X_i)=\emptyset,$
a contradiction again. Therefore, $E(G_{j_1})$ only contains $e_3$ as well. On the whole, when $d\geq4$, $\bigcup_{i\in[2,d-1]}E(G_i)=\{e_3\}$, and if $|\bigcup_{i\in[2,d-1]}E(G_i)|\geq 2$, then $d=3$.

If $\bigcup_{i\in[2,d-1]}E(G_i)=\{e_3\}$, then for every $i\in[2,d-1]$,
$$\mathcal F_2^{(i)}=\{A_i=([k-1]\setminus\{i\})\cup e_3\}.$$
For any $F_3\in\mathcal F_3$, since $F_3\cap F_{1,1}\cap(\bigcap_{i\in[2,d-1]}A_i)=F_3\cap\{k+1\}\neq\emptyset$
and
$F_3\cap F_{1,2}\cap(\bigcap_{i\in[2,d-1]}A_i)=F_3\cap\{k+2\}\neq\emptyset$, we have $e_3=\{k+1,k+2\}\subseteq F_3$. Thus
$$\mathcal F_3=\{F\in\mathcal F:[d-1]\subseteq F,F\cap[d,k-1]=\emptyset,\{k+1,k+2\}\subseteq F\}.$$
After interchanging the elements $k$ and $k+2$, one can check that $\mathcal F\subseteq\mathcal G(k,d)$.

If $|\bigcup_{i\in[2,d-1]}E(G_i)|\geq 2$, then $d=3$, and so $|E(G_2)|\geq 2$. On the other hand, since $2=|\mathcal F_{2}^{(1)}|=\max\{|\mathcal F_{2}^{(i)}|:i\in[2]\}$, we have $|\mathcal F_{2}^{(2)}|\leq 2$, i.e., $|E(G_2)|\leq 2$. Thus $|E(G_2)|=2$. Since $E(G_1)=\{e_1,e_2\}$, $G_2$ must contain the edge $e_3$. Without loss of generality, assume that $E(G_2)=\{e_2,e_3\}$. Then
$$\mathcal F_{2}^{(2)}=\{F_{2,1}, F_{2,2}\},\text{ where }F_{2,1}=([k-1]\setminus\{2\})\cup e_2\text{ and }F_{2,2}=([k-1]\setminus\{2\})\cup e_3.$$
For any $F_3\in\mathcal F_3$, since $F_3\cap F_{1,1}\cap F_{2,1}=F_3\cap\{k\}\neq\emptyset$, $F_3\cap F_{1,1}\cap F_{2,2}=F_3\cap\{k+1\}\neq\emptyset$ and $F_3\cap F_{1,2}\cap F_{2,2}=F_3\cap\{k+2\}\neq\emptyset$, we have $[k,k+2]\subseteq F_3$. Therefore,
$$\mathcal F_3=\{F\in\mathcal F:[2]\subseteq F,F\cap[d,k-1]=\emptyset,[k,k+2]\subseteq F\}.$$
After interchanging the elements $k$ and $k+2$, one can check that $\mathcal F\subseteq \mathcal S_1(k,3)$.

\noindent{\bf \underline{Case 2.3.3.3 $|\mathcal F_{2}^{(1)}|=1$.}} Since $|\mathcal F_{2}^{(1)}|=\max\{|\mathcal F_{2}^{(i)}|:i\in[d-1]\}$, $|\mathcal F_{2}^{(i)}|=1$ for every $i\in[d-1]$. Let $\mathcal F_{2}^{(i)}=\{F_{2}^{(i)}\}$. For any $F_3\in\mathcal F_3$, $F_3\cap(\bigcap_{i\in[d-1]}F_{2}^{(i)})=e_1\cap e_2\cap e_3=\emptyset$, a contradiction.
\end{proof}

\section{Proof of Theorem \ref{thm:maximal_intersecting_d+1}}\label{sec:3}

For $\mathcal G\subseteq\binom{[n]}{k}$ and $X\subseteq[n]$ with $|X|<k$, define $d_{\mathcal G}(X)=|\{G\in\mathcal G:X\subseteq G\}|$.

\begin{proof}[{\bf Proof of Theorem \ref{thm:maximal_intersecting_d+1}:}]
Suppose that $\mathcal F\subseteq \binom{[n]}{d+1}$ is a maximal non-trivial $d$-wise intersecting family and $|\mathcal F|>3d(d+1)$. We will show that $\mathcal F$ must be one of $\mathcal H(d+1,d,2)$ and $\mathcal H(d+1,d,3)$ up to isomorphism.
Note that by \eqref{equ:H(k,d,l)}, for any $n>(d+1)^2$, $|\mathcal H(d+1,d,2)|=(d+1)n-d^2-2d>3d(d+1)$ and $|\mathcal H(d+1,d,3)|=3n-2d-4>3d(d+1)$.

Without loss of generality assume that $d_{\mathcal F}([d-1])=\max_{X\in\binom{[n]}{d-1}}d_{\mathcal F}(X)$.
Let $$\mathcal F_1=\{F\in\mathcal F:[d-1]\subseteq F\}.$$
By the non-triviality of $\mathcal F$, $\mathcal F\setminus\mathcal F_1\neq\emptyset$.
If there would be $F\in \mathcal F\setminus\mathcal F_1$ with $|F\cap[d-1]|<d-3$, then for every $F_1\in\mathcal F_1$, $|F\cap F_1|\leq|F\cap[d-1]|+|F_1\setminus[d-1]|<d-1$, a contradiction.
Define
$$\mathcal F_2=\{F\in\mathcal F:|F\cap[d-1]|=d-2\}$$
and
$$\mathcal F_3=\{F\in\mathcal F:|F\cap[d-1]|=d-3\}.$$
Then $\mathcal F=\mathcal F_1\cup\mathcal F_2\cup\mathcal F_3$, and $\mathcal F_1,\mathcal F_2$ and $\mathcal F_3$ are pairwise disjoint.

\noindent\underline{{\bf Case 1 $\mathcal F_3\neq\emptyset$.}}
Assume that $[3,d+3]\in\mathcal F_3$  by renaming the elements in $[n]$.
It follows from Lemma \ref{lem:inter_size} that for any $F_1\in\mathcal F_1$, $|F_1\cap[3,d+3]|=d-3+|F_1\cap[d,d+3]|\geq d-1$, and so $|F_1\cap[d,d+3]|\geq2$.
Furthermore, since $|F_1|=d+1$ and $[d-1]\subseteq F_1$, we have $|F_1\cap[d,d+3]|=2$, and hence $\mathcal F_1\subseteq\{[d-1]\cup Y:|Y\cap[d,d+3]|=2\}$.
This implies that $$d_{\mathcal F}([d-1])=|\mathcal F_1|\leq 6.$$
Let $L=[d-1]\cup Y\in\mathcal F_1$.
For any $F_2\in\mathcal F_2$, by Lemma \ref{lem:inter_size}, $|F_2\cap L|=d-2+|F_2\cap Y|\geq d-1$,
and so $|F_2\cap Y|\geq 1$.
By the assumption of $d_{\mathcal F}([d-1])=\max_{X\in\binom{[n]}{d-1}}d_{\mathcal F}(X)$,
$$|\mathcal F_2|\leq \sum_{W\in\binom{[d-1]}{d-2}}\sum_{u\in Y}d_{\mathcal F}(W\cup\{u\})\leq 2(d-1)d_{\mathcal F}([d-1])\leq12(d-1).$$
For any $F_3\in\mathcal F_3$, $|F_3\cap L|=d-3+|F_3\cap Y|\geq d-1$ by Lemma \ref{lem:inter_size}, and so $F_3$ must contain $Y$ for every $F_3\in\mathcal F_3$.
By the assumption of $d_{\mathcal F}([d-1])=\max_{X\in\binom{[n]}{d-1}}d_{\mathcal F}(X)$,
$$|\mathcal F_3|\leq\sum_{W\in\binom{[d-1]}{d-3}}d_{\mathcal F}(W\cup Y)\leq\binom{d-1}{2}d_{\mathcal F}([d-1])\leq6\binom{d-1}{2}.$$
Therefore, $|\mathcal F|=|\mathcal F_1|+|\mathcal F_2|+|\mathcal F_3|\leq 3d(d+1)$, a contradiction.

\noindent\underline{{\bf Case 2 $\mathcal F_3=\emptyset$.}} In this case $\mathcal F_2\neq\emptyset$.
We construct a hypergraph $H=(V(H),E(H))$, where $V(H)=[d,n]$ and
$$E(H)=\left\{\varepsilon\in\binom{[d,n]}{3}:\varepsilon\subseteq F\in\mathcal F_2\right\}.$$
For each $i\in[d-1]$, let
$$\mathcal F_{2}^{(i)}=\{F\in\mathcal F_2:[d-1]\setminus\{i\}\subseteq F,i\notin F\}$$
and $H_i=(V(H),E(H_i))$, where
$$E(H_i)=\{\varepsilon\in E(H):\varepsilon=F\setminus[d-1], F\in\mathcal F_2^{(i)}\}.$$
Then $\mathcal F_2=\bigcup_{i\in[d-1]}\mathcal F_{2}^{(i)}$. Note that $\mathcal F_2^{(i)}\neq\emptyset$ for any $i\in[d-1]$ because of the non-triviality of $\mathcal F$. There are three cases to consider according to the size of $E(H)$. Before proceeding to the three cases, we give the following simple but useful claim.
\begin{Claim}\label{cla:inter_with_hypergarph}
For any $F_1\in\mathcal F_1$ and any $\varepsilon\in E(H)$, $F_1\cap \varepsilon\neq\emptyset$.
\end{Claim}
\begin{proof}[{\bf Proof of Claim}]
For every $\varepsilon\in E(H)$, there exists $F_2\in\mathcal F_2$ such that $\varepsilon\subseteq F_2$.
For each $F_1\in\mathcal F_1$, since $|F_1\cap F_2|=d-2+|F_1\cap \varepsilon|\geq d-1$ by Lemma \ref{lem:inter_size}, $|F_1\cap \varepsilon|\geq 1$.
\end{proof}

\noindent\underline{{\bf Case 2.1 $|E(H)|=1$.}}
Without loss of generality assume that $E(H)=\{[d,d+2]\}$. Then $\mathcal F_2^{(i)}=\{X_i\}$, where $X_i=[d+2]\setminus\{i\}$ for each $i\in[d-1]$, and so
$$\mathcal F_2=\{X_i:i\in[d-1]\}.$$
For any $F_1\in\mathcal F_1$, by Claim \ref{cla:inter_with_hypergarph}, $F_1\cap[d,d+2]\neq\emptyset$.
Thus
$$\mathcal F_1=\{F\in\mathcal F:[d-1]\subseteq F,F\cap[d,d+2]\neq\emptyset\}.$$
It is readily checked that $\mathcal F\subseteq\mathcal H(d+1,d,3)$.

\noindent\underline{{\bf Case 2.2 $|E(H)|=2$.}}
Let $E(H)=\{\varepsilon_1,\varepsilon_2\}$. Then there exist $1\leq j_1<j_2\leq d-1$ such that $F_{j_1}=([d-1]\setminus\{j_1\})\cup \varepsilon_1\in\mathcal F_2^{(j_1)}$ and $F_{j_2}=([d-1]\setminus\{j_2\})\cup \varepsilon_2\in\mathcal F_2^{(j_2)}$.
By Lemma \ref{lem:inter_size}, $|F_{j_1}\cap F_{j_2}|=d-3+|\varepsilon_1\cap \varepsilon_2|\geq d-1$.
Since $|\varepsilon_1|=|\varepsilon_2|=3$ and $\varepsilon_1\neq \varepsilon_2$, $|\varepsilon_1\cap \varepsilon_2|=2$. Without loss of generality, assume that $\varepsilon_1=\{d,d+1,d+2\}$ and $\varepsilon_2=\{d,d+1,d+3\}$.
Then $$\mathcal F_2=\{F\in\mathcal F:|F\cap[d-1]|=d-2,\{d,d+1\}\subseteq F\}.$$
Take $X_i\in\mathcal F_2^{(i)}$ for $i\in[d-1]\setminus\{j_1,j_2\}$.
For any $F_1\in\mathcal F_1$, we have $F_1\cap F_{j_1}\cap F_{j_2}\cap(\bigcap_{i\in[d-1]\setminus\{j_1,j_2\}}X_i)=F_1\cap\{d,d+1\}\neq\emptyset$.
So $$\mathcal F_1=\{F\in\mathcal F:[d-1]\subseteq F,F\cap\{d,d+1\}\neq\emptyset\}.$$
It is readily checked that $\mathcal F\subseteq\mathcal H(d+1,d,2)$.

\noindent\underline{{\bf Case 2.3 $|E(H)|\geq3$.}}
It follows from Claim \ref{cla:inter_with_hypergarph} that for every $F_1\in\mathcal F_1$, $F_1\setminus[d-1]$ is a cover of $H$, and so $\tau(H)\leq|F_1\setminus[d-1]|=2$.

\noindent\underline{{\bf Case 2.3.1 $\tau(H)=1$.}}
Without loss of generality, assume that $\{d\}$ is a cover of $H$.
Define an auxiliary graph $G=(V(G),E(G))$, where $V(G)=[d+1,n]$ and
$$E(G)=\{e\in\binom{[d+1,n]}{2}:e\cup\{d\}\in E(H)\}.$$
Let $G_i=(V(G),E(G_i))$, where
$$E(G_i)=\{e\in E(G):e\cup\{d\}\in E(H_i)\}.$$
Let $\overline{\mathcal F_1}=\{F\in\mathcal F:d\notin F\}.$
By the non-triviality of $\mathcal F$, $\mathcal F$ contains a set $F$ with $d\notin F$, and so $\overline{\mathcal F_1}\neq\emptyset$.
Since $d\in F_2$ for every $F_2\in\mathcal F_2$, we have
$\overline{\mathcal F_1}\subseteq\mathcal F_1.$
Thus
$$\overline{\mathcal F_1}=\{F\in\mathcal F:[d-1]\subseteq F, d\notin F\}\text{\ \ and\ \ }
\mathcal F_1\setminus\overline{\mathcal F_1}=\{F\in\mathcal F:[d]\subseteq F\}.$$

\noindent\underline{{\bf Case 2.3.1.1 $\tau(G)=1$.}}
Without loss of generality assume that every edge of $G$ contains $d+1$.
Then each hyperedge in $H$ contains $\{d,d+1\}$, and so
$$\mathcal F_2=\{F\in\mathcal F:|F\cap[d-1]|=d-2,\{d,d+1\}\subseteq F\}.$$
Let $\varepsilon_i=\{d,d+1,w_i\}\in E(H)$ for $i\in[3]$, where $w_1,w_2,w_3\in[d+2,n]$ are pairwise distinct.
It follows from Claim \ref{cla:inter_with_hypergarph} that for any $F_1\in\overline{\mathcal F_1}$ and $i\in[3]$, $F_1\cap\varepsilon_i\neq\emptyset$.
Since $d\notin F_1$ and $|F_1|=d+1$, $F_1$ must contain $d+1$.
Thus $$\mathcal F_1=\{F\in\mathcal F:[d-1]\subseteq F, F\cap\{d,d+1\}\neq\emptyset\}.$$
Then $\mathcal F=\mathcal F_1\cup\mathcal F_2\subseteq\mathcal H(d+1,d,2)$.

\noindent\underline{{\bf Case 2.3.1.2 $\tau(G)\geq 2$} and {\bf $\nu(G)=1$.}}
In this case $G$ must be a triangle and $\tau(G)=2$.
Assume, without loss of generality, that $E(G)=\{e_1,e_2,e_3\}$, where
$$e_1=\{d+1,d+2\}, e_2=\{d+1,d+3\}\text{ and }e_3=\{d+2,d+3\}.$$
For any $F_1\in\overline{\mathcal F_1}$ and $i\in[3]$, by Claim \ref{cla:inter_with_hypergarph} together with $d\notin F_1$, $F_1\cap e_i\neq\emptyset$.
Thus for every $F_1\in\overline{\mathcal F_1}$, $F_1$ must contain at least one of $e_1,e_2$ and $e_3$.
Therefore,
$$\overline{\mathcal F_1}\subseteq\{[d-1]\cup e:e\in\{e_1,e_2,e_3\}\}.$$
Without loss of generality, assume that $\overline{F}=[d-1]\cup e_1\in\overline{\mathcal F_1}$.
If there would be $1\leq j_1< j_2\leq d-1$ such that $F_{j_1}=([d]\setminus\{j_1\})\cup e_2\in\mathcal F_2^{(j_1)}$ and $F_{j_2}=([d]\setminus\{j_2\})\cup e_3\in\mathcal F_2^{(j_2)}$, then
$\overline{F}\cap F_{j_1}\cap F_{j_2}\cap(\bigcap_{i\in[d-1]\setminus\{j_1,j_2\}}X_i)=\emptyset,$
where $X_i\in\mathcal F_2^{(i)}$ for $i\in[d-1]\setminus\{j_1,j_2\}$, a contradiction.
Thus $e_2$ and $e_3$ must lie in the same graph $G_j$ for some unique $j\in[d-1]$,
and hence for every $i\in[d-1]\setminus\{j\}$,
$$\mathcal F_{2}^{(i)}=\{A_i\},\text{ where } A_i=[d+2]\setminus\{i\}.$$
Write
$F_{j,1}=([d]\setminus\{j\})\cup e_2\in\mathcal F_2^{(j)}\text{ and }F_{j,2}=([d]\setminus\{j\})\cup e_3\in \mathcal F_2^{(j)}.$
If there would be $L=[d-1]\cup e_2\in\overline{\mathcal F_1}$ (resp. $L=[d-1]\cup e_3\in\overline{\mathcal F_1}$), then $L\cap F_j\cap(\bigcap_{i\in[d-1]\setminus\{j\}}A_i)=\emptyset$, where $F_j=F_{j,2}$ (resp. $F_j=F_{j,1}$), a contradiction. Thus $$\overline{\mathcal F_1}=\{[d-1]\cup e_1\}.$$
For any $L\in\mathcal F_2^{(j)}$, $\overline{F}\cap L\cap(\bigcap_{i\in[d-1]\setminus\{j\}}A_i)=L\cap e_1\neq\emptyset$, and so
$$\mathcal F_2^{(j)}=\{F\in\mathcal F:[d]\setminus\{j\}\subseteq F,j\notin F, F\cap e_1\neq\emptyset\}.$$
Without loss of generality assume that $j=1$.
Then $\mathcal F\subseteq\mathcal G(d+1,d)\cong\mathcal H(d+1,d,3)$.

\noindent\underline{{\bf Case 2.3.1.3 $\nu(G)\geq2$.}}
Without loss of generality, let $$e_1=\{d+1,d+3\}\text{ and }e_2=\{d+2,d+4\}$$ be two disjoint edges of $G$.
If there would be $1\leq j_1< j_2\leq d-1$ such that $F_{j_1}=([d]\setminus\{j_1\})\cup e_1\in\mathcal F_2^{(j_1)}$ and $F_{j_2}=([d]\setminus\{j_2\})\cup e_2\in\mathcal F_2^{(j_2)}$, then $|F_{j_1}\cap F_{j_2}|=d-2$.
Thus $e_1$ and $e_2$ must lie in the same graph $G_j$ for some unique $j\in[d-1]$. Write $$F_{j,1}=([d]\setminus\{j\})\cup e_1\in\mathcal F_2^{(j)}\text{ and }F_{j,2}=([d]\setminus\{j\})\cup e_2\in\mathcal F_2^{(j)}.$$
For any $F_1\in\overline{\mathcal F_1}$, we know $d\notin F_1$. Then by Claim \ref{cla:inter_with_hypergarph}, $F_1\cap e_1\neq\emptyset$ and $F_1\cap e_2\neq\emptyset$, and hence $F_1$ must contain at least one of the following edges:
$$e_3=\{d+1,d+2\},e_4=\{d+3,d+4\},e_5=\{d+2,d+3\}\text{ and }e_6=\{d+1,d+4\}.$$
Thus
$$\overline{\mathcal F_1}\subseteq\{[d-1]\cup e:e\in\{e_3,e_4,e_5,e_6\}\}.$$
Suppose that $\overline{F}=[d-1]\cup e_3\in\overline{\mathcal F_1}$ without loss of generality.
Given $h\in[d-1]\setminus\{j\}$, let $L_h=([d]\setminus\{h\})\cup e\in\mathcal F_2^{(h)}$ for some $e\in E(G)$.
Since $\mathcal F$ is $d$-wise intersecting, for any $ X_i\in\mathcal F_2^{(i)}$ with $i\in[d-1]\setminus\{j,h\}$, we have
$\overline{F}\cap F_{j,1}\cap L_h\cap(\bigcap_{i\in[d-1]\setminus\{j,h\}}X_i)=e\cap\{d+1\}\neq\emptyset$
and
$\overline{F}\cap F_{j,2}\cap L_h\cap(\bigcap_{i\in[d-1]\setminus\{j,h\}}X_i)=e\cap\{d+2\}\neq\emptyset$.
Therefore, for any $h\in[d-1]\setminus\{j\}$,
$$\mathcal F_{2}^{(h)}=\{A_h\},\text{ where }A_h=[d+2]\setminus\{h\}.$$
If there would be $L=[d-1]\cup e_4\in\overline{\mathcal F_1}$, then $L\cap F_{j,1}\cap(\bigcap_{i\in[d-1]\setminus\{j\}}A_i)=\emptyset$, a contradiction.
If there would be $L=[d-1]\cup e_5\in\overline{\mathcal F_1}$ (resp. $L=[d-1]\cup e_6\in\overline{\mathcal F_1}$), then $L\cap F_{j}\cap(\bigcap_{i\in[d-1]\setminus\{j\}}A_i)=\emptyset$, where $F_j=F_{j,1}$ (resp. $F_j=F_{j,2}$).
This yields a contradiction again.
Thus $$\overline{\mathcal F_1}=\{[d-1]\cup e_3\}.$$
For each $F'\in\mathcal F_2^{(j)}$, $\overline{F}\cap F'\cap(\bigcap_{h\in[d-1]\setminus\{j\}}A_h)=F'\cap e_3\neq\emptyset$.
Thus $$\mathcal F_2^{(j)}=\{F\in\mathcal F:[d]\setminus\{j\}\subseteq F,j\notin F, F\cap e_3\neq\emptyset\}.$$
Without loss of generality assume that $j=1$.
It is readily checked that $\mathcal F\subseteq\mathcal G(d+1,d)\cong\mathcal H(d+1,d,3)$.

\noindent\underline{{\bf Case 2.3.2 $\tau(H)=2$.}} Suppose without loss of generality that $$\overline{F}=[d+1]\in\mathcal F_1.$$
By Claim \ref{cla:inter_with_hypergarph}, $\{d,d+1\}$ is a cover of the hypergraph $H$.
Since $\tau(H)=2$, $H$ contains two hyperedges $\varepsilon_1,\varepsilon_2$ with the property that $d\in\varepsilon_1,d+1\notin\varepsilon_1$ and $d\notin\varepsilon_2,d+1\in\varepsilon_2$.

If there would be $1\leq j_1<j_2\leq d-1$ such that $F_{j_1}=([d-1]\setminus\{j_1\})\cup\varepsilon_1\in\mathcal F_2$ and $F_{j_2}=([d-1]\setminus\{j_2\})\cup \varepsilon_2\in\mathcal F_2$, then $\overline{F}\cap F_{j_1}\cap F_{j_2}\cap(\bigcap_{i\in[d-1]\setminus\{j_1,j_2\}}X_i)=\emptyset$, where $X_i\in\mathcal F_2^{(i)}$ for $i\in[d-1]\setminus\{j_1,j_2\}$, a contradiction.
Hence $\varepsilon_1,\varepsilon_2$ must lie in the same hypergraph $H_j$ for some unique $j\in[d-1]$.
Without loss of generality take $j=1$, and write $$F_{1,1}=[2,d-1]\cup\varepsilon_1\in\mathcal F_2^{(1)}\text{ and }F_{1,2}=[2,d-1]\cup\varepsilon_2\in \mathcal F_2^{(1)}.$$

Given $s\in[2,d-1]$, if there would exist $L\in\mathcal F_2^{(s)}$ such that $d\notin L$ (resp. $d+1\notin L$), then for $X_i\in\mathcal F_2^{(i)}$ with $i\in[2,d-1]\setminus\{s\}$ and $F_1=F_{1,1}$ (resp. $F_1=F_{1,2}$), $\overline{F}\cap L\cap F_{1}\cap(\bigcap_{i\in[2,d-1]\setminus\{s\}}X_i)=\emptyset$, a contradiction.
Thus for any $A_s\in\mathcal F_2^{(s)}$, $\{d,d+1\}\subseteq A_s$, and by Lemma \ref{lem:inter_size},  $|A_s\cap\varepsilon_1|=|A_s\cap F_{1,1}|-(d-3)\geq 2$ and $|A_s\cap\varepsilon_2|=|A_s\cap F_{1,2}|-(d-3)\geq 2.$
Therefore, for any $s\in[2,d-1]$,
\begin{align}\label{eqn:6-24}
\mathcal F_{2}^{(s)}=\{F\in\mathcal F:([d+1]\setminus\{s\})\subseteq F, s\notin F, |F\cap\varepsilon_1|\geq2, |F\cap\varepsilon_2|\geq2\}.
\end{align}

By Lemma \ref{lem:inter_size}, $|F_{1,1}\cap F_{1,2}|=d-2+|\varepsilon_1\cap\varepsilon_2|\geq d-1$, and so $|\varepsilon_1\cap\varepsilon_2|\geq1$.

\noindent\underline{{\bf Case 2.3.2.1 $|\varepsilon_1\cap\varepsilon_2|=1$.}}
Without loss of generality assume that $\varepsilon_1\cap\varepsilon_2=\{d+2\}$.
For every $s\in[2,d-1]$,
$$\mathcal F_2^{(s)}=\{A_s\},\text{ where }A_s=[d+2]\setminus\{s\},$$
which yields $E(H_s)=\{[d,d+2]\}$ for each $s\in[2,d-1]$. By Claim \ref{cla:inter_with_hypergarph}, for any $F_1\in\mathcal F_1$, $|F_1\cap[d,d+2]|\geq 1$.
If there would be $L\in\mathcal F_1$ such that $|L\cap[d,d+2]|=1$, then $L=[d-1]\cup\{x,y\}$, where $x\in[d,d+2]$ and $y\notin[d,d+2]$.
By Claim \ref{cla:inter_with_hypergarph}, $\{x,y\}$ would be a cover of $H$, and so there were $\varepsilon\in E(H)$ such that $y\in\varepsilon$ and $x\notin\varepsilon$. Thus $\varepsilon$ would lie in $H_1$. Let $F_{1,3}=[2,d-1]\cup\varepsilon\in\mathcal F_2^{(1)}$.
A contradiction occurs since $L\cap F_{1,3}\cap(\bigcap_{i\in[2,d-1]}A_i)=\emptyset$.
Therefore, for each $F\in\mathcal F_1$, $|F\cap [d,d+2]|\geq 2$, and
$$\mathcal F_1=\{F\in\mathcal F:[d-1]\subseteq F, |F\cap [d,d+2]|\geq 2\}.$$

For any $A_1\in\mathcal F_2^{(1)}$ and $s\in [2,d-1]$, it follows from Lemma \ref{lem:inter_size} that $|A_1\cap[d,d+2]|=|A_1\cap A_s|-(d-3)\geq 2$. Thus
$$\mathcal F_2^{(1)}=\{F\in\mathcal F:1\notin F,[2,d-1]\subseteq F,|F\cap[d,d+2]|\geq 2\}.$$


Let $\pi$ be a permutation on $[n]$ that interchanges the elements $1$ and $d+2$, and fixes all the other elements of $[n]$. Then for any $s\in[2,d-1]$, $\pi(\mathcal F_2^{(s)})=\mathcal F_2^{(s)}$,
\begin{equation}
\begin{aligned}
\pi(\mathcal F_2^{(1)})&=\{F\in\pi(\mathcal F):[d-1]\subseteq F,F\cap\{d,d+1\}\neq\emptyset,d+2\notin F\}\\
&\cup\{F\in\pi(\mathcal F):1\notin F,[2,d-1]\subseteq F,\{d,d+1\}\subseteq F,d+2\notin F\}\notag
\end{aligned}
\end{equation}
and
\begin{equation}
\begin{aligned}
\pi(\mathcal F_1)&=\{F\in\pi(\mathcal F):[d-1]\subseteq F,F\cap\{d,d+1\}\neq\emptyset,d+2\in F\}\\
&\cup\{F\in\pi(\mathcal F):1\notin F,[2,d-1]\subseteq F,\{d,d+1\}\subseteq F,d+2\in F\}.\notag
\end{aligned}
\end{equation}
It is readily checked that $\pi(\mathcal F)\subseteq\mathcal H(d+1,d,2)$.

\noindent\underline{{\bf Case 2.3.2.2 $|\varepsilon_1\cap \varepsilon_2|=2$.}}
Without loss of generality assume that $\varepsilon_1\cap \varepsilon_2=\{d+2,d+3\}$. Then  $$\varepsilon_1=\{d,d+2,d+3\}\text{\ \ and\ \ }\varepsilon_2=\{d+1,d+2,d+3\}.$$
Given $s\in[2,d-1]$, by \eqref{eqn:6-24}, for any $A_s=([d-1]\setminus\{s\})\cup\varepsilon'\in\mathcal F_2^{(s)}$, $\varepsilon'$ must be one of $\varepsilon_3$ and $\varepsilon_4$, where
$$\varepsilon_3=\{d,d+1,d+2\}\text{\ \ and\ \ }\varepsilon_4=\{d,d+1,d+3\}.$$
Thus $|E(H_s)\cap\{\varepsilon_3,\varepsilon_4\}|\geq1$ for any $s\in[2,d-1]$ and  $|E(H)\cap\{\varepsilon_3,\varepsilon_4\}|\geq1$.

If $\{\varepsilon_3,\varepsilon_4\}\subseteq E(H)$, then for any $F_1\in\mathcal F_1$, by Claim \ref{cla:inter_with_hypergarph}, $|F_1\cap\varepsilon_1|\geq1$, $|F_1\cap\varepsilon_2|\geq1$, $|F_1\cap\varepsilon_3|\geq1$ and $|F_1\cap\varepsilon_4|\geq1$, and so $\mathcal F_1\subseteq\{[d-1]\cup e:e\in\binom{[d,d+3]}{2}\}$, which implies
$$d_{\mathcal F}([d-1])=|\mathcal F_1|\leq 6.$$
Given $[d-1]\cup Y\in\mathcal F_1$, for any $F_2\in\mathcal F_2$, by Lemma \ref{lem:inter_size}, $|F_2\cap Y|=|F_2\cap([d-1]\cup Y)|-(d-2)\geq 1$. Thus by the assumption of $d_{\mathcal F}([d-1])=\max_{X\in\binom{[n]}{d-1}}d_{\mathcal F}(X)$, we have
$$|\mathcal F_2|\leq \sum_{W\in\binom{[d-1]}{d-2}}\sum_{u\in Y}d_{\mathcal F}(W\cup\{u\})\leq 2(d-1)d_{\mathcal F}([d-1])\leq12(d-1).$$
Then $|\mathcal F|=|\mathcal F_1|+|\mathcal F_2|=12d-6<3d(d+1)$, a contradiction.

If $|E(H)\cap\{\varepsilon_3,\varepsilon_4\}|=1$, then without loss of generality assume that $\varepsilon_3\in E(H)$ and $\varepsilon_4\notin E(H)$.
For any $s\in[2,d-1]$, since $|E(H_s)\cap\{\varepsilon_3,\varepsilon_4\}|\geq1$,
$$\mathcal F_2^{(s)}=\{A_s\}, \text{ where }A_s=([d-1]\setminus\{s\})\cup\varepsilon_3.$$
By Claim \ref{cla:inter_with_hypergarph}, $|F_1\cap\varepsilon_3|\geq 1$ for any $F_1\in\mathcal F_1$.
If there would be $F_1\in\mathcal F_1$ such that $|F_1\cap\varepsilon_3|=1$, then $F_1=[d-1]\cup\{x,y\}$, where $x\in\varepsilon_3$ and $y\notin\varepsilon_3$.
By Claim \ref{cla:inter_with_hypergarph}, $\{x,y\}$ would be a cover of $H$, and so there were $\varepsilon\in E(H)$ such that $y\in\varepsilon$ and $x\notin\varepsilon$. Thus $\varepsilon$ would lie in $H_1$. Let $F_{1,3}=[2,d-1]\cup\varepsilon\in\mathcal F_2^{(1)}$. Then a contradiction occurs since $F_1\cap F_{1,3}\cap(\bigcap_{i\in[2,d-1]}A_i)=\emptyset$.
Therefore for every $F_1\in\mathcal F_1$, $|F_1\cap\varepsilon_3|\geq 2$. Furthermore, since $[d-1]\subseteq F_1$, we have $|F_1\cap\varepsilon_3|=2$, and so
$$d_{\mathcal F}([d-1])=|\mathcal F_1|\leq 3.$$
Given $[d-1]\cup Y\in\mathcal F_1$, for any $A_1\in\mathcal F_2^{(1)}$, by Lemma \ref{lem:inter_size}, $|A_1\cap Y|=|A_1\cap([d-1]\cup Y)|-(d-2)\geq 1$.
Thus by the assumption of $d_{\mathcal F}([d-1])=\max_{X\in\binom{[n]}{d-1}}d_{\mathcal F}(X)$,
$$|\mathcal F_2^{(1)}|\leq\sum_{u\in Y}d_{\mathcal F}([2,d-1]\cup\{u\})\leq\sum_{u\in Y}d_{\mathcal F}([d-1])\leq 6.$$
Recall that $\mathcal F=\mathcal F_1\cup\mathcal F_2$ and $\mathcal F_2=\bigcup_{s\in[d-1]}\mathcal F_2^{(s)}$.
Then $|\mathcal F|=|\mathcal F_1|+|\mathcal F_2^{(1)}|+\sum_{s\in[2,d-1]}|\mathcal F_2^{(s)}|\leq 3+6+d-2=d+7<3d(d+1)$, a contradiction.
\end{proof}

\section{Comparison of the sizes of extremal families}\label{sec:inequalities}

This section is devoted to comparing the sizes of extremal families in Theorems \ref{thm:maximal_intersecting_d+2} and \ref{thm:maximal_intersecting_d+1} that are all maximal non-trivial $d$-wise intersecting families. We will prove Corollaries \ref{cor:first_second_max_inter}--\ref{cor:sixth_max_inter_d=3} which characterize the largest and the second largest maximal non-trivial $d$-wise intersecting $k$-uniform families for any $k>d\geq 3$, the third largest and the fourth largest maximal non-trivial $d$-wise intersecting $k$-uniform families for any $k>d+1\geq 4$, and the fifth largest and the sixth largest maximal non-trivial $3$-wise intersecting  $k$-uniform families for any $k\geq 5$.

\begin{Lemma}\label{lem:size_G(k,d)}
Assume that $d\geq 3$.
\begin{enumerate}
\item[$(1)$] If $k\geq d+1$ and $n\geq 2k-d+1$, then $|\mathcal H(k,d,k-d+1)|>|\mathcal G(k,d)|$.
\item[$(2)$] If $k\geq d+4$ and $n>\max\{2k+6,24(d-1)\}$, then $|\mathcal H(k,d,k-d)|<|\mathcal G(k,d)|$. If $k\in\{d+2,d+3\}$ and $n\geq 2k-d+2$, then $|\mathcal H(k,d,k-d)|>|\mathcal G(k,d)|$.
\end{enumerate}
\end{Lemma}

\begin{proof}
(1) If $k\geq d+1\geq 4$ and $n\geq 2k-d+1$, then by \eqref{equ:H(k,d,l)} and \eqref{equ:G(k,d)},
$|\mathcal H(k,d,k-d+1)|-|\mathcal G(k,d)|=\binom{n-k-1}{k-d}+(d-3)(n-k-1)>0.$

(2) By \eqref{equ:H(k,d,l)} and \eqref{equ:G(k,d)}, $|\mathcal H(k,d,k-d)|-|\mathcal G(k,d)|=(d-1)\binom{n-k+1}{2}-\binom{n-k-1}{k-d-1}-2(n-k)-d+3.$

If $k\geq d+4\geq 7$ and $n\geq 2k-d+3$, then
$|\mathcal H(k,d,k-d)|-|\mathcal G(k,d)|\leq (d-1)\binom{n-k+1}{2}-\binom{n-k-1}{3}-2(n-k)-d+3<(d-1)\binom{n-k+1}{2}-\binom{n-k-1}{3}$. Since $$\frac{\binom{n-k-1}{3}}{(d-1)\binom{n-k+1}{2}}=\frac{(n-k-1)(n-k-2)(n-k-3)}{3(d-1)(n-k+1)(n-k)}>\frac{(n-k-3)^3}{3(d-1)n^2}
>\frac{n}{24(d-1)}>1,$$
where the penultimate inequality holds when $n>2k+6$ and the last inequality holds when $n>24(d-1)$, we have $|\mathcal H(k,d,k-d)|<|\mathcal G(k,d)|$ for $n>\max\{2k+6,24(d-1)\}$.

If $k\in\{d+2,d+3\}$ and $n\geq 2k-d+2$, then $|\mathcal H(k,d,k-d)|-|\mathcal G(k,d)|\geq(d-1)\binom{n-k+1}{2}-\binom{n-k-1}{2}-2(n-k)-d+3=(d-2)\binom{n-k+1}{2}-d+2>0$.
\end{proof}

\begin{Lemma}\label{lem:size_S(k,3)}
\begin{enumerate}
\item[$(1)$] If $k\geq 5$ and $n\geq 2k-2$, then $|\mathcal G(k,3)|>|\mathcal S_1(k,3)|>|\mathcal S(k,3)|$. Furthermore, if $k\geq6$ and $n\geq 2k-2$, then $|\mathcal S(k,3)|>|\mathcal S_3(k,3)|=|\mathcal S_2(k,3)|$.
\item[$(2)$] If $5\leq k\leq 7$ and $n\geq 2k-1$, then $|\mathcal H(k,3,k-3)|>|\mathcal S_1(k,3)|$.
\item[$(3)$] If $k=8$ and $n\geq 22$, then $|\mathcal H(k,3,k-3)|<|\mathcal S(k,3)|$.
\item[$(4)$] If $k\geq 9$ and $n>\max\{48,2k+10\}$, then $|\mathcal H(k,3,k-3)|<|\mathcal S_2(k,3)|$. If $k=8$ and $n\geq 11$, then $|\mathcal H(k,3,k-3)|>|\mathcal S_2(k,3)|$.
\end{enumerate}
\end{Lemma}
\begin{proof}
(1) If $k\geq 5$ and $n\geq 2k-2$, then by (\ref{equ:G(k,d)}) and (\ref{equ:S1(k,3)}), $|\mathcal G(k,3)|-|\mathcal S_1(k,3)|=\binom{n-k-2}{k-4}+n-k-3>0$, and by (\ref{equ:S(k,3)}) and (\ref{equ:S1(k,3)}), $|\mathcal S_1(k,3)|-|\mathcal S(k,3)|=n-k-3>0$. If $k\geq 6$ and $n\geq 2k-2$, then by (\ref{equ:S(k,3)}) and (\ref{equ:S2(k,3)_S3(k,3)}), $|\mathcal S(k,3)|-|\mathcal S_2(k,3)|=|\mathcal S(k,3)|-|\mathcal S_3(k,3)|=\binom{n-k-3}{k-5}>0$.

(2) If $5\leq k\leq 7$ and $n\geq 2k-1$, then by (\ref{equ:H(k,d,l)}) and (\ref{equ:S1(k,3)}), $|\mathcal H(k,3,k-3)|-|\mathcal S_1(k,3)|=2\binom{n-k+1}{2}-\binom{n-k-2}{k-5}-(n-k+3)\geq 2\binom{n-k+1}{2}-\binom{n-k-2}{2}-(n-k+3)
=\binom{n-k+1}{2}+2(n-k-3)>0$.

(3) If $n\geq 22$, then by (\ref{equ:H(k,d,l)}) and \eqref{equ:S(k,3)}, $|\mathcal H(8,3,5)|-|\mathcal S(8,3)|=2\binom{n-7}{2}-\binom{n-10}{3}-6<0$.

(4) If $k\geq 9$ and $n\geq 2k$, then by (\ref{equ:H(k,d,l)}) and (\ref{equ:S2(k,3)_S3(k,3)}), $|\mathcal H(k,3,k-3)|-|\mathcal S_2(k,3)|=2\binom{n-k+1}{2}-\binom{n-k-3}{k-6}-6<2\binom{n-k+1}{2}-\binom{n-k-3}{3}$. Since
$$\frac{\binom{n-k-3}{3}}{2\binom{n-k+1}{2}}=\frac{(n-k-3)(n-k-4)(n-k-5)}{6(n-k+1)(n-k)}>\frac{(n-k-5)^3}{6n^2}
>\frac{n}{48}>1,$$
where the penultimate inequality holds when $n>2k+10$ and the last inequality holds when $n>48$, we have  $|\mathcal H(k,3,k-3)|<|\mathcal S_2(k,3)|$ for $n>\max\{48,2k+10\}$.
If $k=8$ and $n\geq 11$, $|\mathcal H(8,3,5)|-|\mathcal S_2(8,3)|=2\binom{n-7}{2}-\binom{n-11}{2}-6>0$.
\end{proof}

\begin{Lemma}\label{lem:size_H(k,d,l)}
Let $d\geq 3$. If $3\leq l\leq k-d+1$ and $n>(d-1)^{1/(l-2)}(k-d)+k+2$, then $|\mathcal H(k,d,l+1)|>|\mathcal H(k,d,l)|$.
\end{Lemma}

\begin{proof} By \eqref{equ:H(k,d,l)}, $|\mathcal H(k,d,l+1)|-|\mathcal H(k,d,l)|=\binom{n-d-l}{k-d}-(d-1)\binom{n-d-l}{k-d-l+2}$.
Since $\frac{a}{b}>\frac{a+1}{b+1}$ for $a>b>0$, we have
$$\frac{\binom{n-d-l}{k-d}}{(d-1)\binom{n-d-l}{k-d-l+2}}=
\frac{(n-k-2)\cdots(n-k-l+1)}{(d-1)(k-d)\cdots(k-d-l+3)}
>\frac{1}{d-1}(\frac{n-k-2}{k-d})^{l-2}>1.$$
Therefore, $|\mathcal H(k,d,l+1)|>|\mathcal H(k,d,l)|$.
\end{proof}

\begin{Lemma}\label{lem:size_A(k,d)}
Assume that $d\geq 3$.
\begin{enumerate}
\item[$(1)$] If $3\leq l\leq \min\{d+1,k-d+2\}$, then $|\mathcal H(k,d,2)|>|\mathcal H(k,d,l)|$. If $d+2\leq l\leq k-d+2$ and $n>d(d-1)(k-d)2^{k-d-1}$, then $|\mathcal H(k,d,2)|<|\mathcal H(k,d,l)|$.
\item[$(2)$] If $d+1\leq k\leq 2d+1$ and $n\geq 3d+3$, $|\mathcal H(k,d,2)|>|\mathcal G(k,d)|$.
\end{enumerate}
\end{Lemma}

\begin{proof}
(1) For $l\geq 3$,
$|\mathcal H(k,d,2)|-|\mathcal H(k,d,l)|=(d-1)\left(\binom{n-d-1}{k-d}-\binom{n-d-l+1}{k-d-l+2}\right)-\sum_{i=3}^{l}\binom{n-d-i+1}{k-d}.$

If $3\leq l\leq \min\{d+1,k-d+2\}$,
\begin{align*}
&\ |\mathcal H(k,d,2)|-|\mathcal H(k,d,l)| \\
\geq &\ (d-1)\binom{n-d-1}{k-d}-\sum_{i=3}^{d+1}\binom{n-d-i+1}{k-d}-(d-1)\binom{n-d-l+1}{k-d-l+2}\\
= &\ \sum_{i=3}^{d+1}\left(\binom{n-d-1}{k-d}-\binom{n-d-i+1}{k-d}\right)-(d-1)\binom{n-d-l+1}{k-d-l+2}.
\end{align*}
Notice that for $i\in[4,d+1]$,
$\binom{n-d-1}{k-d}-\binom{n-d-i+1}{k-d}>\binom{n-d-2}{k-d-1}\geq\binom{n-d-l+1}{k-d-l+2}$, and for $i=3$, $\binom{n-d-1}{k-d}-\binom{n-d-2}{k-d}=\binom{n-d-2}{k-d-1}\geq\binom{n-d-l+1}{k-d-l+2}$. Therefore, $|\mathcal H(k,d,2)|>|\mathcal H(k,d,l)|$.

If $d+2\leq l\leq k-d+2$, then
\begin{align*}
&\ |\mathcal H(k,d,2)|-|\mathcal H(k,d,l)|\\
= &\ \sum_{i=d+2}^{2d}\sum_{j=d+2}^{i}\binom{n-j}{k-d-1}-\sum_{i=d+2}^l\binom{n-d-i+1}{k-d}-(d-1)\binom{n-d-l+1}{k-d-l+2}\\
< &\ \frac{d(d-1)}{2}\binom{n-d-2}{k-d-1}-\binom{n-d-l+1}{k-d}.\notag
\end{align*}
Since
$$\frac{2\binom{n-d-l+1}{k-d}}{d(d-1)\binom{n-d-2}{k-d-1}}
>\frac{2(n-k-l+2)^{k-d}}{d(d-1)(k-d)n^{k-d-1}}>\frac{n}{d(d-1)(k-d)2^{k-d-1}}>1,$$
where the penultimate inequality holds when $n>4k-2d\geq 2k+2l-4$ and the last inequality holds when $n>d(d-1)(k-d)2^{k-d-1}$, we have $|\mathcal H(k,d,2)|<|\mathcal H(k,d,l)|$ for $n>d(d-1)(k-d)2^{k-d-1}$ that is larger than $4k-2d$.

(2) If $k=d+1$, it follows from Lemma \ref{lem:size_G(k,d)} that $|\mathcal H(d+1,d,2)|>|\mathcal G(d+1,d)|$ for $n\geq d+3$. Assume that $d+2\leq k\leq 2d+1$. By \eqref{equ:H(k,d,l)} and \eqref{equ:G(k,d)},
$
|\mathcal H(k,d,2)|-|\mathcal G(k,d)|
= (d-1)\binom{n-d-1}{k-d}-\sum_{i=d+2}^{k-1}\binom{n-i}{k-d}-\binom{n-k-1}{k-d-1}-2(n-k)-(d-3).
$
If $k=2d+1$, then 
\begin{align*}
&\ |\mathcal H(2d+1,d,2)|-|\mathcal G(2d+1,d)| \\
= &\ \sum_{i=d+2}^{2d}\sum_{j=d+2}^{i}\binom{n-j}{d}-\binom{n-2d-2}{d}-2n+3d+5 \\
\geq &\ \sum_{i=d+2}^{d+3}\sum_{j=d+2}^{i}\binom{n-j}{d}-\binom{n-2d-2}{d}-2n+3d+5 \\
> &\ 2\binom{n-d-2}{d}-2n+3d+5 > 4(n-d-2)-2n+3d+5=2n-d-3>0,
\end{align*}
where the last two inequalities hold when $n\geq 3d+3$.
If $d+2\leq k\leq 2d$, then $k-d-2\leq d-2$, and hence
\begin{align*}
& |\mathcal H(k,d,2)|-|\mathcal G(k,d)| \\
\geq &\ \binom{n-d-1}{k-d}+\sum_{i=d+2}^{k-1}\left(\binom{n-d-1}{k-d}-\binom{n-i}{k-d}\right)-
\binom{n-k-1}{k-d-1}-2(n-k)-(d-3) \\
= & \ \binom{n-d-1}{k-d} +\sum_{i=d+2}^{k-1}\sum_{j=d+2}^i \binom{n-j}{k-d-1} - \binom{n-k-1}{k-d-1} - 2(n-k)-(d-3) \\
\geq & \ \binom{n-d-1}{k-d} - 2(n-k)-(d-3) > \binom{n-d-1}{k-d} - 2n > \binom{n-d-1}{2} - 2n >0,
\end{align*}
where the last two inequalities hold when $n\geq 3d+3$. Therefore, $|\mathcal H(k,d,2)|>|\mathcal G(k,d)|$ for $4\leq d+1\leq k\leq 2d+1$ and $n\geq 3d+3$.
\end{proof}

\begin{proof}[{\bf Proof of Corollaries \ref{cor:first_second_max_inter}--\ref{cor:sixth_max_inter_d=3}:}]
If $k=d+1\geq 4$ and $n>(d+1)^2$, it follows from Theorem \ref{thm:maximal_intersecting_d+1} that if $\mathcal F\subseteq\binom{[n]}{d+1}$ is a maximal non-trivial $d$-wise intersecting family and $|\mathcal F|>3d(d+1)$, then $\mathcal F\cong\mathcal H(d+1,d,2)$ or $\mathcal H(d+1,d,3)$. By Lemma \ref{lem:size_A(k,d)}$(1)$, $|\mathcal H(d+1,d,2)|>|\mathcal H(d+1,d,3)|$. Therefore, Corollary \ref{cor:first_second_max_inter} holds for $k=d+1$.

Assume that $k\geq d+2$ and $n>n_1(k,d)=d+2(k-d)^2(k^{k-d}-1)^kk!$. Let $\mathcal F\subseteq\binom{[n]}{k}$ be a maximal non-trivial $d$-wise intersecting family with $|\mathcal F|>(k-d-\frac{1}{2})\binom{n-d}{k-d}$.

\noindent\underline{{\bf Case 1 $k=d+2$.}} If $d\geq 4$, it follows from Theorem \ref{thm:maximal_intersecting_d+2} that $\mathcal F$ is isomorphic to one of $\mathcal H(d+2,d,2)$, $\mathcal H(d+2,d,3)$, $\mathcal H(d+2,d,4)$ and $\mathcal G(d+2,d)$.
By Lemma \ref{lem:size_A(k,d)}$(1)$, $|\mathcal H(d+2,d,2)|>|\mathcal H(d+2,d,4)|$.
For $n>n_1(d+2,d)>3d+2$, by Lemma \ref{lem:size_H(k,d,l)}, $|\mathcal H(d+2,d,4)|>|\mathcal H(d+2,d,3)|$.
For $n>n_1(d+2,d)>d+5$, by Lemma \ref{lem:size_G(k,d)}$(1)$, $|\mathcal H(d+2,d,3)|>|\mathcal G(d+2,d)|$.
Therefore,
$|\mathcal H(d+2,d,2)|>|\mathcal H(d+2,d,4)|>|\mathcal H(d+2,d,3)|>|\mathcal G(d+2,d)|,$
and hence Corollaries \ref{cor:first_second_max_inter}, \ref{cor:third_max_inter} and \ref{cor:fourth_max_inter} hold for $k=d+2$ and $d\geq 4$.

If $d=3$ and $k=5$, it follows from Theorem \ref{thm:maximal_intersecting_d+2} that $\mathcal F$ is isomorphic to one of $\mathcal H(5,3,2)$, $\mathcal H(5,3,3)$, $\mathcal H(5,3,4)$, $\mathcal G(5,3)$, $\mathcal S(5,3)$ and $\mathcal S_1(5,3)$. For $n>n_1(5,3)$, using the same argument as in the above paragraph, we can see that $|\mathcal H(5,3,2)|>|\mathcal H(5,3,4)|>|\mathcal H(5,3,3)|>|\mathcal G(5,3)|$.
For $n>n_1(5,3)>8$, $|\mathcal G(5,3)|>|\mathcal S_1(5,3)|>|\mathcal S(5,3)|$ by lemma \ref{lem:size_S(k,3)}$(1)$.
Therefore, Corollaries \ref{cor:first_second_max_inter}--\ref{cor:sixth_max_inter_d=3} hold for $k=5$ and $d=3$.

\noindent\underline{{\bf Case 2 $d+3\leq k\leq 2d$.}} If $d=3$ and $k=2d=d+3=6$, It follows from Theorem \ref{thm:maximal_intersecting_d+2} that $\mathcal F$ is isomorphic to one of $\mathcal H(6,3,2)$, $\mathcal H(6,3,3)$, $\mathcal H(6,3,4)$, $\mathcal H(6,3,5)$, $\mathcal G(6,3)$, $\mathcal S(6,3)$, $\mathcal S_1(6,3)$, $\mathcal S_2(6,3)$ and $\mathcal S_3(6,3)$. Since $n>n_1(6,3)$, Lemma \ref{lem:size_A(k,d)}$(1)$ implies $|\mathcal H(6,3,5)|>|\mathcal H(6,3,2)|>|\mathcal H(6,3,4)|$,
Lemma \ref{lem:size_H(k,d,l)} implies $|\mathcal H(6,3,4)|>|\mathcal H(6,3,3)|$, Lemma \ref{lem:size_G(k,d)}$(2)$ implies $|\mathcal H(6,3,3)|>|\mathcal G(6,3)|$, and Lemma \ref{lem:size_S(k,3)}$(1)$ implies $|\mathcal G(6,3)|>|\mathcal S_1(6,3)|>|\mathcal S(6,3)|>|\mathcal S_2(6,3)|=|\mathcal S_3(6,3)|$. Therefore, Corollaries \ref{cor:first_second_max_inter}--\ref{cor:sixth_max_inter_d=3} hold for $k=6$ and $d=3$.

Assume that $d\geq4$. It follows from Theorem \ref{thm:maximal_intersecting_d+2} that $\mathcal F$ is isomorphic to one of $\mathcal H(k,d,2)$, $\mathcal H(k,d,k-d)$, $\mathcal H(k,d,k-d+1)$, $\mathcal H(k,d,k-d+2)$ and $\mathcal G(k,d)$. Since $n>n_1(k,d)$, by Lemma \ref{lem:size_H(k,d,l)},
$$|\mathcal H(k,d,k-d+2)|>|\mathcal H(k,d,k-d+1)|>|\mathcal H(k,d,k-d)|.$$
For $n>n_1(k,d)>2k-d+1$, by Lemma \ref{lem:size_G(k,d)}$(1)$, $|\mathcal H(k,d,k-d+1)|>|\mathcal G(k,d)|$, and hence $|\mathcal H(k,d,k-d+1)|>\max\{|\mathcal G(k,d)|,|\mathcal H(k,d,k-d)|\}$.
By Lemma \ref{lem:size_A(k,d)}$(1)$, due to $k-d+1\leq d+1$, $|\mathcal H(k,d,2)|>|\mathcal H(k,d,k-d+1)|$, and hence $|\mathcal H(k,d,k-d+1)|<\min\{|\mathcal H(k,d,2)|,|\mathcal H(k,d,k-d+2)|\}$.
Thus $\mathcal H(k,d,k-d+1)$ is the third largest maximal non-trivial $d$-wise intersecting family, and Corollary \ref{cor:third_max_inter} holds for $d+3\leq k\leq 2d$ and $d\geq 4$.

By Lemma \ref{lem:size_A(k,d)}$(1)$, $|\mathcal H(k,d,2)|>|\mathcal H(k,d,k-d+2)|$ for $k\leq2d-1$ and $|\mathcal H(k,d,2)|<|\mathcal H(k,d,k-d+2)|$ for $k=2d$.
Thus when $d+3\leq k\leq 2d-1$, $\mathcal H(k,d,2)$ and $\mathcal H(k,d,k-d+2)$ (resp. when $k=2d$, $\mathcal H(k,d,k-d+2)$ and $\mathcal H(k,d,2)$) are the largest and the second largest maximal non-trivial $d$-wise intersecting family, respectively. Then Corollary \ref{cor:first_second_max_inter} holds for $d+3\leq k\leq 2d$ and $d\geq 4$.

By Lemma \ref{lem:size_G(k,d)}$(2)$, $|\mathcal H(k,d,k-d)|>|\mathcal G(k,d)|$ for $k=d+3$ and $|\mathcal H(k,d,k-d)|<|\mathcal G(k,d)|$ for $k>d+3$.
Therefore, $\mathcal H(k,d,k-d)$ (resp. $\mathcal G(k,d))$ is the fourth largest maximal non-trivial $d$-wise intersecting family for $k=d+3$ (resp. $k>d+3$). Then Corollary \ref{cor:fourth_max_inter} holds for $d+3\leq k\leq 2d$ and $d\geq 4$.

\noindent\underline{{\bf Case 3 $k=2d+1$.}} If $d\geq4$, it follows from Theorem \ref{thm:maximal_intersecting_d+2} that $\mathcal F$ is isomorphic to one of $\mathcal H(2d+1,d,2)$, $\mathcal H(2d+1,d,d+1)$, $\mathcal H(2d+1,d,d+2)$, $\mathcal H(2d+1,d,d+3)$ and $\mathcal G(2d+1,d)$. For $n>n_1(2d+1,d)$, Lemma \ref{lem:size_H(k,d,l)} implies $|\mathcal H(2d+1,d,d+3)|>|\mathcal H(2d+1,d,d+2)|$, Lemma \ref{lem:size_A(k,d)}$(1)$ implies $|\mathcal H(2d+1,d,d+2)|>|\mathcal H(2d+1,d,2)|$, Lemma \ref{lem:size_A(k,d)}$(2)$ implies $|\mathcal H(2d+1,d,2)|>|\mathcal G(2d+1,d)|$, and Lemma \ref{lem:size_G(k,d)}$(2)$ implies $|\mathcal G(2d+1,d)|>|\mathcal H(2d+1,d,d+1)|$.
Therefore,
$|\mathcal H(2d+1,d,d+3)|>|\mathcal H(2d+1,d,d+2)|>|\mathcal H(2d+1,d,2)|>|\mathcal G(2d+1,d)|>|\mathcal H(2d+1,d,d+1)|,$
and hence Corollaries \ref{cor:first_second_max_inter}--\ref{cor:fourth_max_inter} hold for $k=2d+1$ and $d\geq 4$.

If $d=3$, it follows from Theorem \ref{thm:maximal_intersecting_d+2} that $\mathcal F$ is isomorphic to one of $\mathcal H(7,3,2)$, $\mathcal H(7,3,4)$, $\mathcal H(7,3,5)$, $\mathcal H(7,3,6)$, $\mathcal G(7,3)$, $\mathcal S(7,3)$, $\mathcal S_1(7,3)$, $\mathcal S_2(7,3)$ and $\mathcal S_3(7,3)$.
Using the same argument as in the above paragraph, we have
$|\mathcal H(7,3,6)|>|\mathcal H(7,3,5)|>|\mathcal H(7,3,2)|>|\mathcal G(7,3)|>|\mathcal H(7,3,4)|.$
Since $n>n_1(7,3)$, Lemma \ref{lem:size_S(k,3)}$(2)$ implies $|\mathcal H(7,3,4)|>|\mathcal S_1(7,3)|$ and Lemma \ref{lem:size_S(k,3)}$(1)$ implies $|\mathcal S_1(7,3)|>|\mathcal S(7,3)|>|\mathcal S_2(7,3)|=|\mathcal S_3(k,3)|$.
Therefore, Corollaries \ref{cor:first_second_max_inter}--\ref{cor:sixth_max_inter_d=3} hold for $k=2d+1$ and $d=3$.

\noindent\underline{{\bf Case 4 $k>2d+1$.}} If $d\geq4$, it follows from Theorem \ref{thm:maximal_intersecting_d+2} that $\mathcal F$ is isomorphic to one of $\mathcal H(k,d,k-d)$, $\mathcal H(k,d,k-d+1)$, $\mathcal H(k,d,k-d+2)$ and $\mathcal G(k,d)$.
Since $n>n_1(k,d)$, Lemma \ref{lem:size_H(k,d,l)} implies $|\mathcal H(k,d,k-d+2)|>|\mathcal H(k,d,k-d+1)|$ and Lemma \ref{lem:size_G(k,d)} implies $|\mathcal H(k,d,k-d+1)|>|\mathcal G(k,d)|>|\mathcal H(k,d,k-d)|$.
Therefore, Corollaries \ref{cor:first_second_max_inter}--\ref{cor:fourth_max_inter} hold for $k>2d+1$ and $d\geq 4$.

Assume that $d=3$. It follows from Theorem \ref{thm:maximal_intersecting_d+2} that $\mathcal F$ is isomorphic to one of $\mathcal H(k,3,k-3)$, $\mathcal H(k,3,k-2)$, $\mathcal H(k,3,k-1)$, $\mathcal G(k,3)$, $\mathcal S(k,3)$, $\mathcal S_1(k,3)$, $\mathcal S_2(k,3)$ and $\mathcal S_3(k,3)$.
Since $n>n_1(k,3)$, Lemma \ref{lem:size_H(k,d,l)} implies $|\mathcal H(k,3,k-1)|>|\mathcal H(k,3,k-2)|$, Lemma \ref{lem:size_G(k,d)}$(1)$ implies $|\mathcal H(k,3,k-2)|>|\mathcal G(k,3)|$, and Lemma \ref{lem:size_S(k,3)}$(1)$ implies $|\mathcal G(k,3)|>|\mathcal S_1(k,3)|>|\mathcal S(k,3)|$.
Thus
$|\mathcal H(k,3,k-1)|>|\mathcal H(k,3,k-2)|>|\mathcal G(k,3)|>|\mathcal S_1(k,3)|>|\mathcal S(k,3)|.$
If $k=8$, by Lemma \ref{lem:size_S(k,3)} $(3)$ and $(4)$,
$|\mathcal S(8,3)|>|\mathcal H(8,3,5)|>|\mathcal S_2(8,3)|=|\mathcal S_3(8,3)|,$
and so Corollaries \ref{cor:first_second_max_inter}--\ref{cor:sixth_max_inter_d=3} hold for $k=8$ and $d=3$. If $k\geq9$, by Lemma \ref{lem:size_S(k,3)} $(1)$ and $(4)$,
$|\mathcal S(k,3)|>|\mathcal S_2(k,3)|=|\mathcal S_3(k,3)|>|\mathcal H(k,3,k-3)|,$
and so Corollaries \ref{cor:first_second_max_inter}--\ref{cor:sixth_max_inter_d=3} hold for $k\geq 9$ and $d=3$.
\end{proof}

\section{Concluding remarks}

This paper investigates the structure of maximal non-trivial $d$-wise intersecting $k$-uniform families with large sizes for any $k>d\geq 3$. Theorem \ref{thm:maximal_intersecting_d+2} and Theorem \ref{thm:maximal_intersecting_d+1} examine the cases of $k\geq d+2$ and $k=d+1$, respectively.

It is required that $n>n_1(k,d)$ in Theorem \ref{thm:maximal_intersecting_d+2} because of using the Sunflower Lemma (see Lemma \ref{lem:Sunflower} and Lemma \ref{lem:size_B^d}). A natural question is how to reduce the value of $n_1(k,d)$. Balogh and Linz \cite{BL}, recently, gave a short proof to greatly improve the lower bound for $n$ in Corollary \ref{cor:first_second_max_inter}(1) by using the simple fact that every non-trivial $d$-wise intersecting family is $(d-1)$-intersecting. Since the largest non-trivial $(d-1)$-intersecting families have been investigated systematically by Ahlswede and Khachatrian in \cite{AK96,AK97,AK99} and they are exactly $d$-wise intersecting, as a simple corollary, Balogh and Linz showed that a non-trivial $d$-wise intersecting $k$-uniform family with the largest size is isomorphic to one of $\mathcal H(k,d,2)$ and $\mathcal H(k,d,k-d+2)$ for any $n>(1+\frac{d}{2})(k-d+2)$.

Using the same idea presented by Balogh and Linz and combining the recent work by Cao, Lv and Wang \cite{CLW} on determining the structure of maximal non-trivial $(d-1)$-intersecting families $\mathcal F\subseteq \binom{[n]}{k}$ with $|\mathcal F|\geq (k-d+1)\binom{n-d}{k-d}-\binom{k-d+1}{2}\binom{n-d-1}{k-d-1}\sim (k-d+1)\binom{n}{k-d}$, we can give an improvement of the lower bound for $n$ in Corollary \ref{cor:first_second_max_inter}(2) when $k\geq d+2$. Cao, Lv and Wang provided $k-d+3$ maximal non-trivial $(d-1)$-intersecting families, three of which are $d$-wise intersecting. As a corollary of their result, we can see that a maximal non-trivial $d$-wise intersecting $k$-uniform family with the second largest size is isomorphic to one of $\mathcal H(k,d,2)$, $\mathcal H(k,d,k-d+2)$ and $\mathcal H(k,d,k-d+1)$ for any $n>d-2+\max\{\binom{d+1}{2},\frac{k-d+3}{2}\}(k-d+2)^2$.

However, it seems complicated to obtain more maximal non-trivial $(d-1)$-intersecting families that are also ($d-1$)-wise intersecting by employing the method in \cite{CLW}. Therefore, a problem is how to improve the lower bound for $n$ in Corollaries \ref{cor:third_max_inter}--\ref{cor:sixth_max_inter_d=3}. Note that the order of magnitude of the lower bound for the size of $\mathcal F$ in Corollaries \ref{cor:third_max_inter}--\ref{cor:sixth_max_inter_d=3} is $(k-d)\binom{n}{k-d}$.

Theorem \ref{thm:maximal_intersecting_d+1} implies that for any $n>(d+1)^2$, there is no maximal non-trivial $d$-wise intersecting family $\mathcal F\subseteq\binom{[n]}{d+1}$ when $3d(d+1)<|\mathcal F|<\min\{|\mathcal H(d+1,d,2)|,|\mathcal H(d+1,d,3)|\}=\min\{(d+1)n-d^2-2d,3n-2d-4\}$. This result cannot be derived by Cao, Lv and Wang's work \cite{CLW}.

Finally we remark that O'Neill and Verstra\"{e}te \cite{OV} conjectured that Theorem \ref{thm:non_d_wise_intersecting} holds for $n>kd/(d-1)$. Balogh and Linz \cite{BL} presented counterexamples to this conjecture, and made a new conjecture \cite[Conjecture 3]{BL} for the structure of the largest non-trivial $d$-wise intersecting families for $n\leq(1+\frac{d}{2})(k-d+2)$.
Tokushige \cite{To22} presented counterexamples to the conjecture of Balogh and Linz, and made a new one \cite[Problem 2]{To22}. Subsequently, Balogh and Linz proposed a modified version of Tokushige's conjecture \cite[Question 4]{BL}.

\subsection*{Acknowledgements}
The authors thank  the anonymous referees for their valuable comments and suggestions that helped improve the equality of the paper.

\appendix
\setcounter{equation}{0}
\renewcommand\theequation{A.\arabic{equation}}

\section{Appendix: Lower bounds for the sizes of extremal families}\label{Sec:B}

This appendix gives lower bounds for the sizes of the families $\mathcal H(k,d,l)$, $\mathcal G(k,d)$, ${\cal S}(k,3)$, ${\cal S}_1(k,3)$, ${\cal S}_2(k,3)$ and ${\cal S}_3(k,3)$. We will see that their sizes are all greater than $(k-d-\frac{1}{2})\binom{n-d}{k-d}$ except for the case of $\mathcal H(k,d,2)$ with $k\geq 2d+2$, and hence satisfy the condition of Theorem \ref{thm:maximal_intersecting_d+2}. The following combinatorial identity will be often used:
\begin{equation}
\begin{aligned}\label{equ:comb_indetity}
\binom{n}{j}-\binom{n-i}{j}=\sum_{h=1}^{i}\binom{n-h}{j-1},
\end{aligned}
\end{equation}
which can be obtained by applying $\binom{a}{b}=\binom{a-1}{b}+\binom{a-1}{b-1}$ repeatedly.

\begin{Lemma}\label{lem:size_H(k,d,l)_and_bound}
Let $k\geq d+2\geq 5$ and $n\geq 2k(k-d)^3+d$. Then
\begin{enumerate}
\item[$(1)$] $|\mathcal H(k,d,2)|>(k-d-\frac{1}{2})\binom{n-d}{k-d}$ when $k<2d+2$, and $|\mathcal H(k,d,2)|<(k-d-\frac{1}{2})\binom{n-d}{k-d}$ when $k\geq 2d+2$;
\item[$(2)$] $|\mathcal H(k,d,l)|>(k-d-\frac{1}{2})\binom{n-d}{k-d}$ when $l\in[k-d,k-d+2]$.
\end{enumerate}
\end{Lemma}

\begin{proof}
(1) Applying (\ref{equ:comb_indetity}), we have
\begin{equation}
\begin{aligned}
|\mathcal H(k,d,2)|-(k-d-\frac{1}{2})\binom{n-d}{k-d}&=(d+\frac{3}{2}-k)\binom{n-d}{k-d}+d\binom{n-d-1}{k-d}\\
&=(2d+\frac{3}{2}-k)\binom{n-d}{k-d}-d\binom{n-d-1}{k-d-1}.\notag
\end{aligned}
\end{equation}
When $k\geq 2d+2$, $|\mathcal H(k,d,2)|<(k-d-\frac{1}{2})\binom{n-d}{k-d}$. When $k<2d+2$, since $n-d\geq 2k(k-d)^3$, $$\frac{(2d+\frac{3}{2}-k)\binom{n-d}{k-d}}{d\binom{n-d-1}{k-d-1}}=\frac{(2d+\frac{3}{2}-k)(n-d)}{d(k-d)}>1,$$ and hence $|\mathcal H(k,d,2)|>(k-d-\frac{1}{2})\binom{n-d}{k-d}$.

(2) When $l=k-d$, applying (\ref{equ:comb_indetity}), we have
\begin{align*}
& |\mathcal H(k,d,k-d)|-(k-d-\frac{1}{2})\binom{n-d}{k-d} \\
= &\ \sum_{i=d}^{k-1}\binom{n-i}{k-d}+(d-1)\binom{n-k+1}{2}-(k-d-\frac{1}{2})\binom{n-d}{k-d} \\
= &\ \frac{1}{2}\binom{n-d}{k-d}+(d-1)\binom{n-k+1}{2}-\sum_{i=d+1}^{k-1}
\sum_{j=d+1}^{i}\binom{n-j}{k-d-1} \\
>&\ \frac{1}{2}\binom{n-d}{k-d}-\sum_{i=d+1}^{k-1}\sum_{j=d+1}^{i}\binom{n-j}{k-d-1} \\
>&\ \frac{1}{2}\binom{n-d}{k-d}-\frac{(k-d)(k-d-1)}{2}\binom{n-d-1}{k-d-1}.\notag
\end{align*}
Since $n\geq 2k(k-d)^3+d$, $|\mathcal H(k,d,k-d)|>(k-d-\frac{1}{2})\binom{n-d}{k-d}$.

If $k-d\geq 3$, then by Lemma \ref{lem:size_H(k,d,l)}, $|\mathcal H(k,d,k-d+2)|>|\mathcal H(k,d,k-d+1)|>|\mathcal H(k,d,k-d)|>(k-d-\frac{1}{2})\binom{n-d}{k-d}$. Thus it remains to examine the case of $k-d=2$ and $l\in\{k-d+1,k-d+2\}=\{3,4\}$. If $k-d=2$ and $l=3$, since
\begin{equation}
\begin{aligned}
|\mathcal H(d+2,d,3)|-\frac{3}{2}\binom{n-d}{2}&=\sum_{i=d}^{d+2}\binom{n-i}{2}+(d-1)(n-k)-\frac{3}{2}\binom{n-d}{2}\\
&=\frac{(n-d-1)(n-d-4)}{4}+\binom{n-d-2}{2}+(d-1)(n-k)>0,\notag
\end{aligned}
\end{equation}
we have $|\mathcal H(k,d,3)|>(k-d-\frac{1}{2})\binom{n-d}{k-d}$. If $k-d=2$ and $l=4$, then apply Lemma \ref{lem:size_H(k,d,l)} to obtain $|\mathcal H(k,d,4)|>|\mathcal H(k,d,3)|>(k-d-\frac{1}{2})\binom{n-d}{k-d}$.
\end{proof}

\begin{Lemma}\label{lem:size_G(k,d)_and_bound}
Let $k\geq d+1\geq4$ and $n\geq k(k-d)^2+d$. Then $|\mathcal G(k,d)|>(k-d-\frac{1}{2})\binom{n-d}{k-d}$.
\end{Lemma}

\begin{proof}
By (\ref{equ:comb_indetity}), $\binom{n-d+1}{k-d+1}-\binom{n-k+1}{k-d+1}=\sum_{i=d}^{k-1}\binom{n-i}{k-d}$, so $|\mathcal G(k,d)|>\sum_{i=d}^{k-1}\binom{n-i}{k-d}$. Since $n\geq k(k-d)^2+d$,
\begin{align*}
\sum_{i=d}^{k-1}\binom{n-i}{k-d}-(k-d-\frac{1}{2})\binom{n-d}{k-d}&
=\sum_{i=d}^{k-1}\left(\binom{n-i}{k-d}-\binom{n-d}{k-d}\right)+\frac{1}{2}\binom{n-d}{k-d}\\
&=\frac{1}{2}\binom{n-d}{k-d}-\sum_{i=d+1}^{k-1}\sum_{j=d+1}^{i}\binom{n-j}{k-d-1}\\
&\geq \frac{1}{2}\binom{n-d}{k-d}-\frac{1}{2}(k-d)(k-d-1)\binom{n-d-1}{k-d-1}>0.\notag
\end{align*}
Therefore, $|\mathcal G(k,d)|>(k-d-\frac{1}{2})\binom{n-d}{k-d}$.
\end{proof}

\begin{Lemma}\label{lem:size_S(k,3)_and_bound}
Let ${\cal S}\in\{{\cal S}(k,3),{\cal S}_1(k,3)\}$ if $k\geq 5$, and ${\cal S}\in\{{\cal S}_2(k,3),{\cal S}_3(k,3)\}$ if $k\geq 6$. Then $|{\cal S}|>(k-\frac{7}{2})\binom{n-3}{k-3}$ for any $n\geq 2k(k-3)^3+3$.
\end{Lemma}

\begin{proof} By \eqref{equ:comb_indetity},
$\binom{n-2}{k-2}-\binom{n-k+1}{k-2}=\sum_{i=3}^{k-1}\binom{n-i}{k-3}$, so $|\mathcal S|>\sum_{i=3}^{k-1}\binom{n-i}{k-3}$. Since $n\geq2k(k-3)^3+3$,
\begin{equation}
\begin{aligned}
\sum_{i=3}^{k-1}\binom{n-i}{k-3}-(k-\frac{7}{2})\binom{n-3}{k-3}&
=\sum_{i=3}^{k-1}\left(\binom{n-i}{k-3}-\binom{n-3}{k-3}\right)+\frac{1}{2}\binom{n-3}{k-3} \\
&=\frac{1}{2}\binom{n-3}{k-3}-\sum_{i=4}^{k-1}\sum_{j=4}^{i}\binom{n-j}{k-4}\\
&>\frac{1}{2}\binom{n-3}{k-3}-
\frac{(k-3)(k-4)}{2}\binom{n-4}{k-4}>0.\notag
\end{aligned}
\end{equation}
Therefore, $|{\cal S}|>(k-\frac{7}{2})\binom{n-3}{k-3}$.
\end{proof}


\begin{thebibliography}{99}
\bibitem{AK96}
R.~Ahlswede and L.H.~Khachatrian, The complete nontrivial-intersection theorem for systems of finite sets, J. Combin. Theory Ser. A, 76 (1996), 121--138.

\bibitem{AK97}
R.~Ahlswede and L.H.~Khachatrian, The complete intersection theorem for systems of finite sets, European J. Combin., 18 (1997), 125--136.

\bibitem{AK99}
R.~Ahlswede and L.H.~Khachatrian, A pushing-pulling method: new proofs of intersection theorems, Combinatorica, 19 (1999), 1--15.


\bibitem{BL}
J.~Balogh and W.~Linz, Short proofs of three results about intersecting systems, 2021, arXiv:2104.00778.

\bibitem{CLW}
M.~Cao, B.~Lv, and K.~Wang, The structure of large non-trivial $t$-intersecting families of finite sets, European J. Combin., 97 (2021), 103373.


\bibitem{DEF}
M.~Deza, P.~Erd\H{o}s, and P.~Frankl, Intersection properties of systems of finite sets, Proc. Lond. Math. Soc., 36 (1978), 369--384.

\bibitem{EKR}
P.~Erd\H{o}s, C.~Ko and R.~Rado, Intersection theorems for systems of finite sets, Quart. J. Math. Oxford Ser., 12 (1961), 313--320.

\bibitem{ER}
P.~Erd\H{o}s and R.~Rado, A combinatorial theorem, J. London Math. Soc., 25 (1950), 249--255.

\bibitem{ER60}
P.~Erd\H{o}s and R.~Rado, Intersection theorems for systems of sets, J. London Math. Soc., 35 (1960), 85--90.

\bibitem{F76a}
P.~Frankl, On Sperner families satisfying an additional condition, J. Combin. Theory Ser. A, 20 (1976), 1--11.


\bibitem{F78}
P.~Frankl, On intersecting families of finite sets, J. Combin. Theory Ser. A, 24 (1978), 146--161.

\bibitem{F80}
P.~Frankl, On intersecting families of finite sets, Bull. Austral. Math. Soc., 21 (1980), 363--372.

\bibitem{F87}
P.~Frankl, The shifting technique in extremal set theory, Surveys in Combinatorics 1987, London Math. Soc. Lecture Note Ser., vol. 123, Cambridge Univ. Press, Cambridge, 1987, 81--110.

\bibitem{HK}
J.~Han and Y.~Kohayakawa, The maximum size of a non-trivial intersecting uniform family that is not a subfamily of
the Hilton-Milner family, Proc. Amer. Math. Soc., 145 (2017), 73--87.

\bibitem{HM}
A.~Hilton and E.C.~Milner, Some intersection theorems for systems of finite sets, Quart. J. Math. Oxford Ser., 18 (1967), 369--384.

\bibitem{HP}
Y.~Huang and Y.~Peng, Stability of intersecting families, 2022, arXiv:2205.05394.

\bibitem{Kamat}
V.~Kamat, Stability analysis for $k$-wise intersecting families, Electron. J. Combin., 18 (2011), \#P115.

\bibitem{KM}
A.~Kostochka and D.~Mubayi, The structure of large intersecting families, Proc. Amer. Math. Soc., 145 (6) (2017), 2311--2321.

\bibitem{Ku}
A.~Kupavskii, Structure and properties of large intersecting families, 2018, arXiv:1810.00920.

\bibitem{OV}
J.~O'Neill and J.~Verstra\"{e}te, Non-trivial $d$-wise intersecting families, J. Combin. Theory Ser. A, 178 (2021), 105369.




\bibitem{To22}
N.~Tokushige, The maximum measure of non-trivial 3-wise intersecting families, Math. Program., (2023). https://doi.org/10.1007/s10107-023-01969-x.

\end{thebibliography}
\end{document}